\documentclass[11pt,reqno]{amsart}
\usepackage{fullpage}
\usepackage{amsmath,amsthm,verbatim,amssymb,amsfonts,amscd, graphicx,bbm}
\usepackage{graphics}
\usepackage{mathtools}
\usepackage{mathrsfs}
\usepackage{eufrak}
\usepackage{esint}

\usepackage{hyperref}

\newtheorem{theorem}{Theorem}
\newtheorem{proposition}[theorem]{Proposition}
\newtheorem{lemma}[theorem]{Lemma}
\newtheorem{corollary}[theorem]{Corollary}

\theoremstyle{definition}
\newtheorem{remark}[theorem]{Remark}

\newcommand{\cref}[1]{Corollary~\ref{c.#1}}

\numberwithin{equation}{section}
\numberwithin{theorem}{section}

\newcommand{\Z}{\mathbb{Z}}
\newcommand{\N}{\mathbb{N}}
\newcommand{\R}{\mathbb{R}}

\newcommand{\bT}{\mathbb{T}}

\newcommand{\cS}{\mathcal{S}}
\newcommand{\cD}{\mathcal{D}}

\newcommand{\cK}{\mathcal{K}}

\newcommand{\cL}{\mathcal{L}}

\newcommand{\pv}{\mathrm{p.v.}}
\newcommand{\ran}{\mathrm{ran}}

\DeclareMathOperator{\tr}{tr}

\renewcommand{\tilde}{\widetilde}
\renewcommand{\div}{\mathrm{div}\,}
\newcommand{\ol}{\overline}

\newcommand{\eps}{\varepsilon}

\title{Layer potentials for Lam\'e systems and homogenization of perforated elastic medium with clamped holes}%Quantitative homogenization for Lam\'e systems in periodically perforated domains with clamped boundary}

\author{Wenjia Jing}
\address{Yau Mathematical Sciences Center,
  Tsinghua University,
  No.1 Tsinghua Yuan,
  Beijing 100084, People's Republic of China}
\email{wjjing@tsinghua.edu.cn}

\date{\today}

\begin{document}

\begin{abstract}

We investigate Lam\'e systems in periodically perforated domains, and establish quantitative homogenization results in the setting where the domain is clamped at the boundary of the holes. Our method is based on layer potentials and it provides a  unified proof for various regimes of hole-cell ratios (the ratio between the size of the holes and the size of the periodic cells), and, more importantly, it yields natural correctors that facilitate error estimates. A key ingredient is the  asymptotic analysis for the rescaled cell problems, and this is studied by exploring the convergence of the periodic layer potentials for the Lam\'e system to those in the whole space when the period tends to infinity.

\smallskip

\noindent{\bf Key words}: periodic homogenization, perforated domain, Lam\'e systems, layer potentials, oscillating test function method.

\smallskip

\noindent{\bf Mathematics subject classification (MSC 2010)}: 35B27, 35J08

\end{abstract}

\maketitle

%%%%%%%%%%%%%%
%%%%%%%%%%%%%%
\section{Introduction}

In this paper we are motivated to establish the quantitative homogenization results for the elastostatic problem in a periodically perforated domain where the deformation of the material is prescribed at the boundary of the holes. Let $D^\eps = D^{\eps,\eta} \subseteq \R^d$, $d\ge 2$, model the perforated elastic medium,  obtained by removing a periodic array of identical holes. $\eps$ is the typical distance between neighboring holes, and  $\eta\eps$ is the length scale of each hole; $\eta$ in general depends on $\eps$. Mathematically, the homogenization problem corresponds to the asymptotic analysis of the following Lam\'e system. 
\begin{equation}
\label{eq:hetlame}
\left\{
\begin{aligned}
&-\mathcal{L}^{\lambda,\mu} [u](x) = f(x), &\qquad x\in D^\eps,\\%[u^\eps](x) := -\left[\mu\Delta u + (\lambda+\mu)\nabla (\nabla \cdot u^\eps)\right](x) = f(x), &\qquad x\in D^\eps,\\
&u^\eps(x) = 0, &\qquad x \in \partial D^\eps.
\end{aligned}
\right.
\end{equation}
Here, $u^\eps : D^\eps \to \R^d$ is a vector field modeling the displacement field of the material reacting to a forcing field $f$. The differential operator $\cL^{\lambda,\mu}$ is given by
\begin{equation}
\mathcal{L}^{\lambda,\mu} := \mu \Delta + (\lambda + \mu) \nabla \nabla \cdot,
\end{equation}
where $\lambda$ and $\mu$ are the so-called Lam\'e parameters. In this paper they are assumed to be constants and satisfy
\begin{equation*}
\mu > 0, \qquad d\lambda + 2\mu > 0.
\end{equation*}
Note that $\lambda + \mu > 0$ always holds. The material occupied by $D^\eps$ is hence homogeneous but porous. In view of the boundary conditions, the porous elastic body $D^\eps$ has prescribed deformations at the holes. The holes can be realized by inclusions whose deformations can be controlled precisely through some mechanism. As we will see, this Dirichlet type boundary conditions result in various asymptotic regimes for \eqref{eq:hetlame} depending on the smallness of $\eta$ relative to $\eps$.

Partial differential equations in porous media, or more generally in domains with heterogenous geometric features, find many applications in applied physics and engineering, e.g.\;in reservoir engineering, environmental studies, material analysis and design, etc. The mathematical studies also attracted many attentions and produced fruitful results. The literature is enormous, and we only mention a few that are closely related to the homogenization of \eqref{eq:hetlame}. In \cite{CioMur-1}, the scalar conductivity problem in perforated domain with Dirichlet condition on the holes was considered, and the authors there first identified the critical smallness of $\eta$ at which the overall effect of the holes emerges in the homogenization limit. In fact, a ``strange term from nowhere'' appears in the effective equation in the critical setting. Error estimates were also obtained in \cite{Kacimi_Murat}. In \cite{Allaire91-1,Allaire91-2}, Allaire established the corresponding theory for Navier-Stokes system, and further clarified, in \cite{Allaire-3}, the relation between the ``Brinkman term" in the critical setting and the conductivity matrix in the Darcy's law, the latter being the effective model when the holes are much larger than the critical size. In \cite{Jing20}, the author developed a new method based on layer potential techniques and established quantitative homogenization for the scalar conductivity problem in a unified manner for various asymptotic regimes. We extend the approach to Lam\'e system in this paper. Some recent related works on homogenization in perforated domains with Dirichlet conditions on the holes can be found in \cite{LuYong19,GiuVel,Ger-Var,feppon:hal-02518528,feppon:hal-02880030}. We remark that when other boundary conditions such as Neumann, Robin or transmission conditions are imposed at the boundary of the holes/inclusions, the asymptotic behavior could be very different; see e.g. \cite{MR548785,AmmGar,AGGJS,J-CMS,Ger-Var}.

\medskip

As in \cite{Jing20}, our unified homogenization approach utilizes the standard oscillating test function method adapted to perforated domains (see \cite{Tartar}). The building blocks of the oscillating test functions are rescaled from the cell problem. In the classical periodic setting, when $\eta$ is fixed, one uses the ansatz
\begin{equation*}
u^\eps(x) = \left[u_0(x,y) + \eps u_1(x,y) + \eps^2 u_2(x,y) + \cdots\right]_{y = \frac{x}{\eps}},
\end{equation*}
and impose that $u_i$ is $\Z^d$-periodic in $y$ and vanishes for $y$ in the holes. Plugging this in \eqref{eq:hetlame}, replacing $\nabla$ by $\frac1\eps \nabla_y + \nabla_x$, we find, formally, $u^\eps/\eps^2 \approx \sum_k \chi_k(\frac{x}{\eps}) f^k(x)$. The vector field $\chi_k$, for each $k = 1,\dots,d$, is the solution to the cell problem
\begin{equation}
\label{eq:chicell}
-\cL^{\lambda,\mu}[\chi_k](y) = e_k \quad \text{in }\, \bT^d\setminus \eta\ol T, \qquad \chi_k = 0 \quad \text{in } \eta T.
\end{equation}
Here and in the sequel, $\bT^d = \R^d/\Z^d$ is the unit flat torus, and $T$ is the model hole. In view of the Riemann-Lebesgue lemma, we expect that the sequence $\frac{u^\eps}{\eps^2}$ converges weakly to $\langle\chi\rangle f$, where columns of $\langle\chi\rangle$ is the average of $\chi_k$'s in the torus. In the general setting considered in \eqref{eq:hetlame}, the holes are of size $\eta_\eps$ when the periodic cell is rescaled to $\bT^d$, \eqref{eq:chicell} hence still depends on $\eps$ through $\eta_\eps$, and we need to address the asymptotic behavior of $\chi_k(\eps)$ as $\eps$ tends to zero. Equivalently, we can rescale the function and define
\begin{equation*}
\chi^\eta_k(x) = \eta^{d-2}\chi_k(\eta x), \qquad x\in \frac1\eta \bT^d.
\end{equation*}
Then we need to consider the problem
\begin{equation}
\label{eq:chieta}
-\cL^{\lambda,\mu}_y [\chi^\eta_k](y) = \eta^d e_k \quad \text{in } \eta^{-1}\bT^d\setminus \ol T, \qquad \chi^\eta_k = 0 \; \text{ in } T,
\end{equation}
where the hole is at the unit scale and the cell is of size $1/\eta$. To establish quantitative homogenization of \eqref{eq:hetlame}, we need to identify the limit of $\chi^\eta_k$, as $\eta \to 0$, and to quantify the convergence rate of appropriate quantities.

Following the idea of \cite{Jing20}, we carry out those asymptotic analysis through an explicit representation of the solution to \eqref{eq:chieta}. This is obtained by using a particular double-layer potential operator which we introduce now. First we recast the Lam\'e system, $-\cL^{\lambda,\mu}[u] = 0$, as a symmetric and strongly elliptic system of the form
\begin{equation*}
-\partial_i \left(A^{\alpha\beta}_{ij} \partial_j u^\beta\right) = f^\alpha, 
\end{equation*}
where summations over $i,j$ and $\beta$ are taken. Symmetry means $A^{\alpha\beta}_{ij} = A^{\beta\alpha}_{ji}$ and ``strongly elliptic'' means:
\begin{equation*}
A^{\alpha\beta}_{ij}\xi^i\xi^j\zeta^\alpha \zeta^\beta > 0 \; \text{for all non-zero vectors}\; \xi = (\xi^i), \, \zeta = (\zeta^\alpha).
\end{equation*}
It turns out that there are in general infinitely many choice for $(A^{\alpha\beta}_{ij})$ with the above constraints. Each choice of $A$ yields a conormal derivative for $u$ on a surface with normal vector $N$, defined by
\begin{equation*}
\left(\frac{\partial u}{\partial \nu_A}\right)^\alpha = N^i A^{\alpha\beta}_{ij}\partial_j u^\beta.
\end{equation*}
Different choices of conormal derivatives induce different definitions of double-layer potentials. The physically most meaningful choice is 
\begin{equation*}
(A^{(1)})^{\alpha\beta}_{ij} = \lambda \delta_{i\alpha}\delta_{j\beta} + \mu(\delta_{ij}\delta_{\alpha\beta} + \delta_{i\beta}\delta_{j\alpha}),
\end{equation*}
which satisfies the additional symmetry $A^{\alpha\beta}_{ij} = A^{i\beta}_{\alpha j} = A^{\alpha j}_{i\beta}$. It results the conormal derivative
\begin{equation*}
\frac{\partial u}{\partial \nu_{A^{(1)}}} = \lambda (\div u) N + 2\mu \epsilon[u]N, \quad \epsilon[u] = \frac12 (\partial_j u^i + \partial_i u^j).
\end{equation*}
In elasticity theory, $\epsilon[u]$ is called the strain tensor and the conormal derivative above corresponds to the normal stress on the surface. In this paper, however, we use a different choice and set
\begin{equation*}
 A^{\alpha\beta}_{ij} = (\lambda+\mu)\delta_{i\alpha}\delta_{j\beta} + \mu \delta_{ij}\delta_{\alpha\beta},
\end{equation*}
or equivalently, we define the conormal derivative
\begin{equation}
\label{eq:conormal}
\frac{\partial u}{\partial \nu} = (\lambda+\nu)(\div u)N + \mu (\nabla u) N.
\end{equation} 
It turns out that the double-layer potential corresponding to \eqref{eq:conormal} (see the definition \eqref{eq:cDT} below) is more convenient to carry out the approach of \cite{Jing20} to Lam\'e systems, because, as we will see, the Green's identity involving this conormal derivative relates to a bilinear form that controls $\nabla u$ rather than $\epsilon[u]$, and this is a stronger control. The resulted jump formulas for the double-layer potential and for the conormal derivative of the single-layer potential, associated to $\partial T$, involve non-compact operators in $L^2(\partial T)$ even when $T$ has smooth boundary. We overcome this difficulty following the work of \cite{MR769382,MR975122}. With clear characterizations of the mapping properties of those operators and of their periodic variants, we can carry out the quantitative homogenization of \eqref{eq:hetlame}.

The rest of paper is organized as follows. In section \ref{sec:main} we set up the backgrounds for perforated domains and for elastostatic layer potentials, and state the main results of the paper. In section \ref{sec:lp} we study the proposed layer potential operators carefully, show that the trace formulas yield Fredholm  operators although compactness is not available, and establish important invertibility results for them and for their periodic variants. We present sufficient details for all $d\ge 2$. In section \ref{sec:cell} we solve \eqref{eq:chieta} using layer potentials, and, taking advantage of the explicit representation, identify the limits and quantify the convergence rates for various quantities involving the rescaled cell problems. Those results are then used in section \ref{sec:qualh} and in section \ref{sec:quanth}, respectively, to establish the qualitative homogenization results and to quantify the convergence rates. We emphasize again that, in this paper, the two dimensional setting is completely covered by the approach, which is  an improvement of \cite{Jing20}.

\smallskip

\noindent{\bfseries Notations.} We list some notations and conventions that are used throughout the paper. We write $x = (x^i)$ for a vector in $\R^d$, and components are always labeled by $i,j,k$ or $\ell$. The standard inner product on $\R^d$ is written as $x\cdot y$ or $\langle x,y\rangle$. For a vector field $u = (u^i)$, its derivative $\nabla u$ is written as a matrix $(\partial_j u^i)$ with row index $i$ and column index $j$; hence, its transpose $(\nabla u)^t$ has elements $\partial_i u^j$. We always use the summation convention, unless otherwise stated, so repeated index is summed over its range. Hence, the matrix-vector product $(\nabla u)N$ is given by $(N^j\partial_j u^i)$. For matrices $A, B$ of the same dimensions, $A:B = a_{ij}b_{ij}$ is the Frobenius inner product, and $|A|$ denotes the Frobenius norm of $A$; the determinant of a square matrix $A$ is written as $\det(A)$ instead. The tensor product of two vectors, $a$ with $b$, is denoted by $a\otimes b$ and has components $a^ib^j$. For vector fields $u,v$ both in $L^2(D)$ or in $L^2(\partial T)$, we use $\langle u,v\rangle_{L^2}$ to denote their inner product in those functional spaces. Let $E$ be a set with finite measure, $\langle u\rangle_E$ and $\fint_E u$ both denote the average of $u$ in $E$, and the subscript $E$ is often omitted when the reference is clear from the context. Finally, for $r > 0$, $rE$ is the rescaled set $\{rx\,:\,x \in E\}$.
%%%%%%%%%
%%%%%%%%%
\section{Preliminaries and main results}
\label{sec:main}

\subsection{Geometric set-ups and assumptions} We first present some details about the perforated domain $D^\eps$ and lay down some main assumptions of the paper.

Let $D \subseteq \R^d$ be an open set. Let $Y = Q_1$ denote the unit cube $(-\frac12,\frac12)^d$, and let $T$ be an open subset of $Y$. We assume that $D$ and $T$ satisfy the following assumptions.
\begin{itemize}
\item[(A1)] The set $D$ is open, bounded and simply connected. $T$ is open and, for simplicity, also simply connected.
\item[(A2)] There is an $\alpha \in (0,1)$, so that the boundaries $\partial T$ and $\partial D$ both are of class $C^{1,\alpha}$. 
\item[(A3)] For some $r_1,r_2$, satisfying $0 < r_1 < r_2 < 1/2$, the set $T$ satisfies 
\begin{equation*}
\ol B_{r_1}(0) \subset T, \qquad \ol T \subset B_{r_2}(0).
\end{equation*}
\end{itemize}
In the rest of the paper, if not further specified, the bounding constant $C$ in all estimates depends only on $d,\lambda,\mu$, and on $T$ and $D$ (through $\alpha$, $r_1,r_2$ and the $C^{1,\alpha}$ characterizations of the boundaries). As usual, the same $C$ is used although its value may change all the time.

Let $Y_f = \ol Y\setminus (\eta \ol T)$, then $Y_f$ denotes the perforated cell at the unit scale and it is connected. We view $Y_f$ as the material part and $\eta T$ the removed hole. Note that the boundary of the cube is included in the material. By tessellation, we obtain $\R^d_f := \cup_{z \in \Z^d} (z+Y_f)$, which is $\R^d$ with a periodic array of copies of $T$ removed. We think $\R^d_f$ as the perforated whole space at the unit scale. By rescaling, we get $\eps \R^d_f$ which is the perforated whole space at the $\eps$-scale. Finally, the perforated domain in \eqref{eq:hetlame} is given by
\begin{equation}
D^\eps = D^{\eps,\eta} = D\cap (\eps \R^d_f).
\end{equation}
We check that $D^\eps$ is connected, and $\partial D^\eps$ consists of $(\partial D)\cap \ol D^\eps$ and $\partial(\eps\R^d_f) \cap D$.

Given $\eps$ and $\eta$, there is a unique weak solution $u^\eps \in H^1_0(D^\eps)$ that solves \eqref{eq:hetlame}, or equivalently, satisfies
\begin{equation}
\label{eq:lamewf}
\int_{D^\eps} \mu \nabla u^\eps : \nabla w + (\lambda+\mu)(\div u^\eps)(\div w) = \int_{D^\eps} f\cdot w, \qquad \forall w \in H^1_0(D^\eps).
\end{equation}
This fact follows from the Lax-Milgram theorem with the help from a special form of Poincar\'e inequality (see Theorem \ref{thm:poincare}). For any function $w \in H^1_0(D^\eps)$, we define $\tilde w$ be the zero-extension
\begin{equation}
\label{eq:zeroext}
\tilde w = w \quad \text{in } \, D^\eps, \qquad \tilde w = 0 \quad \text{in } \,\eps(z+\ol \eta T), \, z\in \Z^d. 
\end{equation}
We use this notation for extension of functions on other perforated domains as well, e.g. on $Y_f$, $\frac1\eta \bT^d\setminus \ol T$ etc., and the extension set zero values inside the holes. 

\medskip

Using $w = u^\eps$ in \eqref{eq:lamewf}, one gets
\begin{equation*}
\mu\|\nabla \tilde u^\eps\|^2_{L^2(D)} + (\lambda+\mu)\|\div \tilde u^\eps\|^2_{L^2(D)} \le \|f\|_{L^2}\|\tilde u^\eps\|_{L^2(D)}.
\end{equation*}
By using the usual Poincar\'e inequality for $\tilde u^\eps \in H^1_0(D)$, we can find $C > 0$ such that 
\begin{equation}
\label{eq:uepsbdd}
\|\nabla \tilde u^\eps\|_{L^2(D)} + \|\tilde u^\eps\|_{L^2(D)} \le C\|f\|_{L^2}.
\end{equation}
On the other hand, if we use the Poincar\'e inequality \eqref{eq:poincare}, we also have
\begin{equation}
\label{eq:uepsbddp}
\|\nabla \tilde u^\eps\|_{L^2(D)} \le C\sigma_\eps \|f\|_{L^2}, \qquad \|\tilde u^\eps\| \le C\sigma_\eps \|f\|_{L^2}^2.
\end{equation}
Here $\sigma_\eps$ is defined by
\begin{equation}
\label{eq:sigeps}
\sigma^2_\eps := \begin{cases}
\eps^2 \eta^{-(d-2)}, \qquad &d\ge 3,\\
\eps^2|\log \eta|, \qquad &d =2.
\end{cases}
\end{equation}
In fact, $\sigma_\eps$ is precisely the bounding constant in \eqref{eq:poincare} when this inequality is applied on each of the $\eps$-cubes contained in $D^\eps$. 

\smallskip

\noindent{\bf Asymptotic regimes.} We identify several asymptotic regimes according to the behavior of the hole-cell ratio $\eta = \eta_\eps$ and the factor $\sigma_\eps$. If $\eta$ converges to a positive constant as $\eps \to 0$, then we are in the classical homogenization setting and the holes occupy a positive volume fraction in the limit. On the other hand, if $\eta = \eta_\eps \to 0$, we say the holes are dilute or their volume fraction is vanishing. 

In this dilute setting, we further identify three sub-cases. If $\sigma_\eps$ converges to a positive number $\sigma_0$ as $\eps \to 0$, we call it the \emph{critical} setting (of hole-cell ratios). In this setting, the size of the holes is critically small compared to the size of cells, which is also the distance of neighboring holes. It is at this critical setting that the asymptotic effect of the holes emerges. If $\sigma_\eps \to \infty$, we call it the \emph{sub-critical} setting; in this case, the holes are of smaller order and their effects can be neglected in the limit. If $\sigma_\eps \to 0$, we call it the \emph{super-critical} setting; the holes are of larger order and their asymptotic effect is more dramatic.

Clearly, \eqref{eq:uepsbddp} is a stronger estimate for the super-critical setting, and \eqref{eq:uepsbdd} is the better one for sub-critical holes. 

\medskip

\subsection{Elastostatic layer potentials} A main ingredient of our analysis is the layer potential theory for Lam\'e systems. It not only provides representations for the solution of \eqref{eq:chieta} but also explains the parameters that enter the effective models for \eqref{eq:hetlame}, for all dilute regimes and for all $d\ge 2$.

Let $e_k$, $k=1,2,\cdots,d$, denote the standard orthonormal basis of $\R^d$. For each $k$, the fundamental solution $\Gamma_k = (\Gamma^j_k)_j$ to the problem
\begin{equation}
\label{eq:GamEq}
\mathcal{L}^{\lambda,\mu}[\Gamma_k] = \mu \Delta \Gamma_k(x) + (\lambda+\mu) \nabla \nabla \cdot \Gamma_k(x) = \delta_0(x) e_k, \qquad \text{in } \R^d,
\end{equation}
subject to decay condition ($d\ge 3$) or logarithmic growth condition ($d=2$), is given by the following explicit formula:
\begin{equation}
\label{eq:Gamma}
\Gamma_{k}^j(x) = \left\{\begin{aligned}
\frac{c_1}{(2-d)\omega_d} \frac{\delta_{jk}}{|x|^{d-2}} - \frac{c_2}{\omega_d} \frac{x^j x^k}{|x|^{d}}, \qquad &d\ge 3,\\
\frac{c_1}{2\pi} (\log|x|)\delta_{jk} - \frac{c_2}{\omega_d} \frac{x^j x^k}{|x|^{d}}, \qquad &d=2,
\end{aligned}
\right.
\end{equation}
where $c_1$ and $c_2$ are two constants defined by
\begin{equation*}
c_1 = \frac{1}{2} \left( \frac{1}{\mu} + \frac{1}{\lambda+2\mu} \right), \quad c_2 = \frac{1}{2} \left( \frac{1}{\mu} - \frac{1}{\lambda+2\mu} \right).
\end{equation*}
The formulas above provide the unique (for $d=2$, up to unimportant additive constants) solution to \eqref{eq:GamEq} with conditions at infinity.

% For each fixed $k$, the gradient $\nabla \Gamma_k = (\partial_i \Gamma_{k}^j)_{ij}$ is given by 
% \begin{equation*}
% \partial_i \Gamma_{k}^j = \frac{c_1 x^i \delta_{jk} }{\omega_d |x|^{d}} - \frac{c_2(x^k\delta_{ij} + x^j\delta_{ik})}{\omega_d |x|^d} + \frac{dc_2}{\omega_d} \frac{x^j x^k}{|x|^{d+1}} \frac{x^i}{|x|},
% \end{equation*}
% and the divergence $\nabla \cdot \Gamma_k = \partial_i \Gamma_{k}^i$ is given by
% \begin{equation*}
%  \partial_i \Gamma_{k}^i = \frac{(c_1-c_2)x^k}{\omega_d |x|^d} = \frac{1}{\lambda+2\mu} \frac{x^k}{\omega_d |x|^d}.
% \end{equation*}

\medskip

Let $T \subseteq \R^d$ be an open set satisfying assumptions (A1) and (A2). The standard \emph{single-layer} potential for Lam\'e system, with momentum $\phi \in L^2(\partial T)$, is defined, through its components, by
\begin{equation}
\label{eq:cST}
(\cS_T[\phi])^k(x) = \int_{\partial T} \Gamma_k(x;y) \cdot \phi(y) dy, \qquad x\in \R^d\setminus \partial T.
\end{equation}
We denote the exterior domain $\R^d\setminus \ol T$ by $T_+$, and, also write $T_- = T$ sometime to emphasize the contrast with $T_+$. It can be checked directly that $\cL^{\lambda,\mu}[\cS_T[\phi]] = 0$ in $T_\pm$. Moreover, $w = \cS_T[\phi]$ is smooth in $T_\pm$ and verifies the decay condition:
\begin{equation}
\label{eq:cSTdecay}
|w(x)| = O(|x|^{-d+2}) \quad \text{for }\,d\ge 3, \qquad |\nabla w(x)| = O(|x|^{-d+1}) \quad \text{for }\, d\ge 2, \qquad \text{as } |x| \to \infty.
\end{equation}
The decay of $|w(x)|$ does not hold for $d=2$ in general, but we have $|w(x)| = O(|x|^{-1})$ at infinity if $\phi \in L^2_0(\partial T)$. Here and in the sequel, $L^2_0(\partial T)$ denotes the subspace of $L^2(\partial T)$ that consists of mean-zero functions.

%%%%%%%

As mentioned in the Introduction, to define double-layer potentials, we need to fix a conormal derivative. Throughout the paper, we adopt \eqref{eq:conormal}. Then for vector fields $u,v$ in $T$ with sufficient regularity, we have the Green's identity
\begin{equation}
\label{eq:GreenI}
\int_{\partial T} v\cdot \frac{\partial u}{\partial \nu} = \int_T \mu \nabla v:\nabla u + (\lambda+\mu)(\nabla\cdot v)(\nabla\cdot u) + \int_T v \cdot \cL^{\lambda,\mu}[u].
\end{equation}
By switching $u$ and $v$, we also have
\begin{equation}
\label{eq:GreenII}
\int_{\partial T} v\cdot \frac{\partial u}{\partial \nu} - u\cdot \frac{\partial v}{\partial \nu} = \int_T v \cdot \cL^{\lambda,\mu}[u] -  u \cdot \cL^{\lambda,\mu}[v].
\end{equation}
Moreover, \eqref{eq:GreenI} still holds on $T_+$, if $|u(x)||\nabla v(x)|$ is of order $o(|x|^{-d+1})$.

Those Green's identities suggest us to define the \emph{double-layer} potential, with momentum $\phi$, by
\begin{equation}
\label{eq:cDT}
(\cD_T[\psi])^k(x) = \int_{\partial T} \frac{\partial \Gamma_k}{\partial \nu_y}(x;y) \cdot \psi(y) dy, \qquad x\in \R^d\setminus \partial T.
\end{equation}
The subscript $\nu_y$ emphasizes that the derivatives in \eqref{eq:conormal} are taken for the $y$-variable. Direct computations on \eqref{eq:Gamma} show that the integral kernel, written as $K(x;y)$ with components $K_{ik}(x;y)$, is given by
\begin{equation*}
\begin{aligned}
K_{ik}(x;y):= \left(\frac{\partial{\Gamma_k}}{\partial \nu_y}(x;y)\right)^i =& -\frac{\mu c_1}{\omega_d} \frac{\langle N_y, x-y\rangle \delta_{ik}}{|x-y|^d} - \frac{d\mu c_2}{\omega_d} \frac{\langle N_y,x-y\rangle (x-y)^i(x-y)^k}{|x-y|^{d+2}} \\
&\;+ \frac{\mu c_2}{\omega_d} \frac{(x-y)^iN_y^k - (x-y)^k N^i_y}{|x-y|^d}.
\end{aligned}
\end{equation*}
Again, $\cD_T[\psi]$ are smooth vector fields and satisfy the homogeneous Lam\'e systems on $T_\pm$. It is also clear that $|\cD_T[\phi]| = O(|x|^{-d+1})$ at infinity, for all $d\ge 2$. 

We use $K(x;y)$, $x,y \in \partial T$, as the integration kernel and define, for $k=1,\dots,d$,
\begin{equation}
\label{eq:cKT}
(\cK_T[\psi])^k(x) = \pv \int_{\partial T} K_{ik}(x;y)\psi^i(y), \qquad x\in \partial T.
\end{equation}
We need to take the principal value integral because of the last term in the formula of $K_{ik}$. In fact, the the other terms are absolutely integrable in $y$ uniformly in $x$, because $\partial T \in C^{1,\alpha}$ implies
\begin{equation}
\label{eq:C1alphaN}
\langle x-y, N_x\rangle \le C|x-y|^{1+\alpha}, \qquad |N_x - N_y| \le C|x-y|^\alpha, \qquad \forall \, x, y \in \partial T.
\end{equation}
Contributions of those terms form a compact operator on $L^2(\partial T)$. The last term, however, is not integrable even for smooth $\partial T$. As a result, $\cK_T$ is a genuine singular integral. Invoking classical theory on singular integrals, namely \cite{MR672839}, we confirm that $\cK_T$ is a bounded linear operator on $L^2(\partial T)$.

%%%%%%%
\medskip

\noindent{\bfseries Trace formulas.} Layer potential operators are useful to solve boundary value problems for Lam\'e systems because their traces on $\partial T$, or more precisely, their non-tangential limits on $\partial T$ from $T_-$ or $T_+$, can be computed. In the sequel, for a function $F$ defined on $T_-$ and $T_+$, we use the notation
\begin{equation*}
F\rvert_\pm(x) = \lim_{t\to 0+} F(x\pm tN_x), \qquad x\in \partial T,
\end{equation*}
provided that the limit exists. In other words, $F\rvert_-$ is the limit from the inside of $T$, and $F\rvert_+$ is the limit from the exterior of $T$. For the single-layer potential defined in \eqref{eq:cST} and for the conormal derivative in \eqref{eq:conormal}, it is known (see \cite{MR975122}) that
\begin{equation}
\label{eq:traceu}
\partial_i u^j_\pm(x) = \pm\left\{\frac{1}{2\mu} N^i_x \phi^j(x) - \gamma_2 N^i_x N^j_x N_x \cdot \phi(x) \right\} 
+ \pv \int_{\partial T} \partial_i \Gamma_{j}^k(x-y) \phi^k(y) dy. 
\end{equation}
Plug this formula in the definition of the conormal derivative, we get
\begin{equation}
\label{eq:traceuN}
\begin{aligned}
\frac{\partial\,\cS_T[\phi]}{\partial \nu}\Big\rvert_\pm (x) &= \pm \frac12 \phi(x) + \pv \int_{\partial T} \left(\mu \partial_j \Gamma^i_k(x-y)N^j_x + (\lambda + \mu)(\div \Gamma_k)(x-y) N^i_x\right) \phi^k(y)\\
&= \pm \frac12 \phi(x) + \cK^*_T[\phi](x).
\end{aligned}
\end{equation}
For the double-layer potential defined in \eqref{eq:cDT}, we have
\begin{equation}
\label{eq:jumpD}
\cD_T[\phi]\rvert_\pm (x) = (\mp \frac12 I + \cK_T)[\phi](x), \qquad \text{in } \partial T.
\end{equation}
In the second line of \eqref{eq:traceuN}, we recognized the integral operator as the adjoint of $\cK_T$ defined in \eqref{eq:cKT}. Indeed, the singular integral operator in the first line of \eqref{eq:traceuN} can be written as
\begin{equation*}
\cK^*_T[\phi](x) = \pv \int_{\partial T} K^*_{ik}(x;y) \phi^i(y),
\end{equation*}
and explicit computation shows
% \begin{equation}
% \begin{aligned}
% K^*_{ik}(x;y) &= \frac{\mu c_1}{\omega_d} \frac{\langle x-y, N_x\rangle \delta_{ik}}{|x-y|^d} + \frac{dc_2 \mu}{\omega_d} \frac{(x-y)^i(x-y)^k \langle x-y,N_x\rangle}{|x-y|^{d+2}} \\
% &\qquad + \frac{\mu c_2}{\omega_d} \frac{(x-y)^i N^k_x - (x-y)^k N^i_x}{|x-y|^d}.
% \end{aligned}
% \end{equation}
% We check that 
\begin{equation*}
K^*_{ik}(x;y) = K_{ki}(y;x).
\end{equation*}
Both $\cK_T$ and $\cK^*_T$ are bounded linear transformations on $L^2(\partial T)$, but they are not compact. Nevertheless, we can compute and check that
\begin{equation}
\label{eq:Kdifference}
\begin{aligned}
K^*_{ik}(x;y) - K_{ik}(x;y) =& \frac{\mu c_1}{\omega_d} \frac{\langle x-y, N_x+N_y\rangle \delta_{ik}}{|x-y|^d} + \frac{dc_2 \mu}{\omega_d} \frac{(x-y)^i(x-y)^k \langle x-y,N_x+N_y\rangle}{|x-y|^{d+2}} \\
&\qquad + \frac{\mu c_2}{\omega_d} \frac{(x-y)^i (N_x-N_y)^k - (x-y)^k (N_x-N_y)^i}{|x-y|^d}.
\end{aligned}
\end{equation}
Thanks to \eqref{eq:C1alphaN}, the function above is integrable in $y$ over $\partial T$, uniformly for $x\in \partial T$. As a result, $\cK^*_T - \cK_T$ is a compact operator on $L^2(\partial T)$. Finally, we also know that $\cS_T[\phi]\rvert_+$ and $\cS_T[\phi]\rvert_-$ agree on $\partial T$, and agree with \eqref{eq:cST} with $x \in \partial T$. Moreover, the tangential derivative of $\cS_T$ on $\partial T$, i.e. the traces of $\tau_x \cdot \nabla \cS_T[\phi]$ from $T_+$ and $T_-$, where $\tau_x$ belongs to the tangent space $X_x(\partial T)$ of $\partial T$ at $x\in \partial T$. This can checked directly from the trace formula \eqref{eq:traceu}.

In section \ref{sec:lp}, we will introduce the periodic variants of the above layer potentials, and use them to solve and analyze \eqref{eq:chieta}.  

%%%%%%%
\subsection{Main results}

The first main result of the paper concerns some mapping properties of the operators $-\frac12 I + \cK_T$ and $-\frac12 I + \cK^*_T$,  which appear in the trace formula \eqref{eq:jumpD}. 

\begin{lemma}
\label{lem:key}
Suppose $d\ge 2$, $T \subseteq \R^d$ is an open bounded set satisfying (A1) and (A2). Then the operators $-\frac12 I + \cK_T$ and $-\frac12 I + \cK^*_T$, as bounded linear transformations on $L^2(\partial T)$, satisfy the following properties.
\begin{itemize}
\item[(1)] The ranges of the operators are closed, and both of their kernels have dimension $d$. Moreover, $\ker(-\frac12 I + \cK_T)$ is the subspace of constant vector fields over $\partial T$.
\item[(2)] The direct sum decomposition $L^2(\partial T) = \ran(-\frac12 I + \cK_T) \oplus \ker(-\frac12 I + \cK_T)$ holds.
\end{itemize}
\end{lemma}
Those results are proved in section \ref{sec:proofkey}. It will be shown that we can find $\phi^*_1,\dots,\phi^*_d$ so that they form a basis for the kernel of $-\frac12 I + \cK^*_T)$, and they satisfy
\begin{equation*}
\int_{\partial T} \phi^*_j = e_j, \qquad -\cS_T[\phi^*_i] = a^*_j \quad\text{on } \ol T, \qquad j = 1,\dots,d.
\end{equation*}
Let $A_T$ be the matrix defined by
\begin{equation}\label{eq:ATdef}
A_T = ((a^*_j)^i) = \begin{bmatrix}
a^*_1 & a^*_2 & \cdots & a^*_d
\end{bmatrix}.
\end{equation}
We will show that $A_T$ is symmetric, and $A_T$ is positive definite for $d\ge 3$. For $d=2$, due to the abnormal rescaling property of $\Gamma_k$ in \eqref{eq:Gamma}, the matrix $A_T$ could be degenerate; however, when the homogenization of \eqref{eq:hetlame} is concerned, we can always assume (see Remark \ref{rem:AT2d}) that
\begin{equation}
\label{eq:AT2d}
\det\, A_T \ne 0.
\end{equation}
The decomposition in item (2) of Lemma \ref{lem:key} is easily done using $\phi^*_j$'s above; see Lemma \ref{lem:L2decom}.

\medskip

Now we state our main results concerning the homogenization of \eqref{eq:hetlame}. We define the matrix
\begin{equation}
\label{eq:Mdef}
M = M_T := \begin{cases}\left.
\begin{aligned}
A^{-1}_T \qquad &\text{if } d\ge 3\\
\frac{c_1}{2\pi} I \qquad &\text{if } d=2
\end{aligned}
\right\} \qquad &\text{in the dilute setting},\\
\smallskip
\left(\fint_Y \chi_j^i\right)^{-1} \qquad &\text{in the classical setting}.
\end{cases}
\end{equation}
In the classical setting, $\eta$ is essentially a fixed parameter, and the problem is in the super-critical setting. The cell problem \eqref{eq:chicell} does not depend on $\eps$, and no further asymptotic analysis is needed. Note that $M$ defined above is positive definite (see Proposition \ref{prop:AT}).

\begin{theorem}
\label{thm:qualh}
Assume $d\ge 2$, assume (A1)(A2) and (A3) holds. For each $\eps \in (0,1)$, let $u^\eps$ be the unique solution of \eqref{eq:hetlame} and $\tilde u^\eps$ be the zero extension, and assume $f \in L^2(D)$. Let $\sigma_\eps$ be defined by \eqref{eq:sigeps}. Then the following holds as $\eps \to 0$.
\begin{itemize}
\item[(1)] In the super-critical setting, i.e. when $\sigma_\eps \to 0$, the zero extension function $\frac{\tilde u^\eps}{\sigma^2_\eps}$ converges weakly to $u$ in $L^2(D)$, with $u = M^{-1}f$.

\item[(2)] In the critical setting, i.e. $\sigma_\eps \to \sigma_0$ for some positive real number $\sigma_0$, the sequence $\tilde u^\eps$ converges weakly in $H^1_0(D)$ to $u$, which is given by the unique solution to the problem
\begin{equation}
\label{eq:homcr}
-\cL^{\lambda,\mu} [u] + \frac{M}{\sigma^2_0} u = f \quad \text{in }\, D, \qquad u = 0 \quad \text{in }\, \partial D.
\end{equation}

\item[(3)] In the sub-critical setting, i.e. $\sigma_\eps \to \infty$, the sequence $\tilde u^\eps$ converges weakly in $H^1_0(D)$ to $u$, which is given by the unique solution to the unperturbed problem
\begin{equation}
\label{eq:homsubc}
-\cL^{\lambda,\mu} [u]  = f \quad \text{in }\, D, \qquad u = 0 \quad \text{in }\, \partial D.
\end{equation}
\end{itemize}
\end{theorem}

The classical setting is included in item (1). It can be proved following the standard arguments in \cite{BLP}. In fact, we show that results in the other settings can be proved following the same arguments, except an additional asymptotic analysis for \eqref{eq:chieta} is needed. Those proofs are presented in section \ref{sec:qualh} below. An advantage of our method is that, it can be quantified relatively easily. This is addressed by the next main theorem.

\begin{theorem}
\label{thm:quanh}
Suppose that the assumptions of Theorem \ref{thm:qualh} hold, and $\eta \to 0$ as $\eps 0$. Let $v^\eps_k$'s be defined by \eqref{eq:vepsdef}. Assume further that the limiting function $u$ of Theorem \ref{thm:qualh}, in each regimes, satisfies: $u \in W^{2,d}(D)$ for $d\ge 3$ and $u \in W^{2,\infty}(D)$ for $d = 2$. Then the following, stated first for $d\ge 3$, holds: 
\begin{itemize}
\item[(1)] In the dilute super-critical setting, there exists $C > 0$ so that for all $\eps$ sufficiently small,
\begin{equation}
\|\frac{\tilde u^\eps}{\sigma^2_\eps} - f^k(x)v^\eps_k(x)\|_{H^1(D)} + \frac{1}{\sigma_\eps} \|\frac{\tilde u^\eps}{\sigma^2_\eps} - f^k(x)v^\eps_k(x)\|_{L^2(D)} \le C(\sigma_\eps + \eta^{\frac{d-2}{2}})\|f\|_{W^{2,d}}
\end{equation}
\item[(2)] In the critical setting, and suppose $\sigma_\eps \to \sigma_0$ for some $\sigma_0 \in (0,\infty)$, then there exists $C > 0$ so that for all $\eps$ sufficiently small,
\begin{equation}
\|\tilde u^\eps - \sigma^2_\eps (\frac{M}{\sigma_0^2} u)^k v^\eps_k \|_{H^1(D)} \le C(\eps + |\sigma_\eps^2 - \sigma^2_0|)\|u\|_{W^{2,d}}.
\end{equation}
\item[(3)] In the sub-critical setting, there exists $C > 0$ so that for all $\eps$ sufficiently small,
\begin{equation}
\|\tilde u^\eps - (M u)^k v^\eps_k \|_{H^1(D)} \le C(\sigma^{-2}_\eps + \eta^{\frac{d-2}{2}})\|u\|_{W^{2,d}}.
\end{equation}
\end{itemize}
For $d=2$, the above results hold with $W^{2,d}$ replaced by $W^{2,\infty}$, and $\eta^{\frac{d-2}{2}}$ replaced by $|\log \eta|^{-\frac12}$.
\end{theorem} 

The quantitative results above contain corrector informations. Take $d\ge 3$ and the sub-critical setting for example, we may write
\begin{equation*}
\tilde u^\eps - (Mu)^k v^\eps_k = \tilde u^\eps - u - r^\eps, \qquad r^\eps := (Mu)^k\left[v^\eps_k - M^{-1}e_k\right].
\end{equation*}
We can think $r^\eps$ as the leading order corrector. Indeed, adding it to $u$, we not only improve the weak convergence of item (3) in Theorem \ref{thm:qualh} to a strong convergence, but can also control the approximation error in $H^1$. Of course, using \eqref{eq:vstrbdd} below which yields estimates for the corrector, we also have the quantitative estimate
\begin{equation*}
\|\tilde u^\eps - u\|_{L^{\frac{2d}{d-2}}} \le C\eta^{\frac{d-2}{2}}\|u\|_{W^{2,d}}
\end{equation*}
We leave such discussions for the other settings to the reader.

Finally, we remark that the $C^{1,\alpha}$ assumption on $\partial T$, in (A2), can be relaxed to $\partial T$ being Lipschitz. We only need to borrow some further techniques of \cite{MR769382,MR975122} to deal with layer potentials on Lipschitz boundaries. Then results in section \eqref{sec:lp} and, hence, the main results of the paper still hold. To simplify the presentations, however, we use the stronger assumption (A2).

%%%%%%%

%%%%%%%
%%%%%%%
\section{Mapping properties for layer-potentials and their periodic variants}
\label{sec:lp}

In this section, we study the properties of the layer potentials and prove Lemma \ref{lem:key}. We also introduce and study their periodic variants, which will be used to analyze \eqref{eq:chieta}.

\subsection{The Rellich's identity}

The scalar version of Lemma \ref{lem:key}, as in \cite{Jing20}, is relatively easy because the Neumann-Poincar\'e operator $\cK_T$ associated to the Laplace operator is compact, for $\partial T \in C^{1,\alpha}$, and Fredholm theory can be invoked. This is not the case for $\cK_T$ in the elastostatic setting, even for smooth $\partial T$. 

To overcome this difficulty, we follow the line of reasoning in \cite{MR769382,MR975122}. An important step is to establish the closedness of the ranges of $-\frac12 I + \cK^*_T$. The key is to show the conormal derivatives of $\cS_T[\phi]$, taken from the two sides of $\partial T$, can bound each other in $L^2$.  To this purpose, we need the following elastostatic version of Rellich formula. Note that $\partial T \in C^{1,\alpha}$ implies, we can find a $C^{1,\alpha}$ vector field $\gamma$ over $\R^d$ with compact support, and for some constant $C > 0$, $\gamma$ satisfies
\begin{equation}
\label{eq:gamfield}
\langle \gamma, N\rangle \ge C > 0, \qquad \text{on } \partial T.
\end{equation}

\begin{proposition}
Let $d\ge 2$, and let $T \subseteq \R^d$ be an open bounded set satisfying (A1) and (A2). Then for any $u$ that verifies $\cL^{\lambda,\mu}[u] = 0$ in $T$ and that $\nabla u$ has trace on $\partial T$, we have    
\begin{equation}
\label{eq:rellich1}
\begin{aligned}
&\int_{\partial T} \langle \gamma,N\rangle \left((\lambda+\mu)(\div u)^2 + \mu|\nabla u|^2\right)\big\rvert_- = 2\int_{\partial T} \langle \gamma,\nabla u \rangle \frac{\partial u}{\partial \nu}\Big\rvert_- \\
&\quad+\int_T (\nabla \cdot \gamma) \left[(\lambda+\mu)(\div u)^2 + \mu|\nabla u|^2\right] -2\int_T (\nabla u\nabla \gamma):\left[(\lambda+\mu)(\div u) I + \mu \nabla u\right].
\end{aligned}
\end{equation}
Similarly, if $\cL^{\lambda,\mu}[u] = 0$ on $T_+$ and $\nabla u$ has trace on $\partial T$, then we have
\begin{equation}
\label{eq:rellich1-1}
\begin{aligned}
&\int_{\partial T} \langle \gamma,N\rangle \left((\lambda+\mu)(\div u)^2 + \mu|\nabla u|^2\right)\big\rvert_+ = 2\int_{\partial T} \langle \gamma,\nabla u \rangle \frac{\partial u}{\partial \nu}\Big\rvert_+ \\
&\quad-\int_T (\nabla \cdot \gamma) \left[(\lambda+\mu)(\div u)^2 + \mu|\nabla u|^2\right] +2\int_T (\nabla u\nabla \gamma):\left[(\lambda+\mu)(\div u) I + \mu \nabla u\right].
\end{aligned}
\end{equation}
\end{proposition}

\begin{proof}
From direct computations, we check that, either in $T$ or in $\R^d\setminus \ol T$,
\begin{equation*}
\nabla \cdot (\gamma|\nabla u|^2) = (\div \gamma) |\nabla u|^2 + 2 [(\gamma\cdot \nabla)\nabla u ]: \nabla u.
\end{equation*}
where the last term is $2 (\partial_j u^k) \gamma^i \partial_i (\partial_j u^k)$ and the summation convention is envoked. On the other hand, using the fact that $u$ satisfies the Lam\'e system, we also have
\begin{equation*}
\begin{aligned}
(\lambda+\mu)\partial_i[(\gamma^j \partial_j u^i)(\div u)] + \mu  \partial_\ell[(\gamma^j \partial_j u^i)(\partial_\ell u^i)] =  \mu\left[(\partial_\ell \gamma^j)(\partial_j u^i)(\partial_\ell u^i) + (\gamma^j\partial_j \partial_\ell u^i)(\partial_\ell u^i)\right]\\
 (\lambda+\mu)\left[(\div u)(\partial_i \gamma^j)(\partial_j u^i) + \frac12 \partial_j(\gamma^j (\div u)^2) - \frac12 (\div \gamma) (\div u)^2 \right].
\end{aligned}
\end{equation*}
The desired equality is then obtained by integrating those identities in $T$ or in $\R^d\setminus \ol T$, using the divergence theorem, and combining the resulted integral identities.
\end{proof}

We can apply the above identities to $u = \cS_T[\phi]$ for a vector field $\phi \in L^2(\partial T)$. For such $u$, using integration by parts and by the jump formula \eqref{eq:traceuN}, we have
\begin{equation}
\label{eq:cSipT}
\int_{\partial T} \frac{\partial u}{\partial \nu}\Big\rvert_-  = 0, \quad \text{and} \quad 
\int_{\partial T} \frac{\partial u}{\partial \nu}\Big\rvert_+  = \int_{\partial T} \phi.
\end{equation}
For $d\ge 3$, in view of the decay condition \eqref{eq:cSTdecay}, we can apply the Green's identity and show
\begin{equation}
\label{eq:greenTm}
\int_{\partial T} u\cdot \frac{\partial u}{\partial \nu}\Big\rvert_-  = \int_T (\lambda+\mu)(\div u)^2 + \mu |\nabla u|^2,
\end{equation}
and
\begin{equation}
\label{eq:greenTp}
\int_{\partial T} u\cdot \frac{\partial u}{\partial \nu}\Big\rvert_+  = -\int_{\R^d \setminus \ol T} (\lambda+\mu)(\div u)^2 + \mu |\nabla u|^2.
\end{equation}
For $d=2$, the identities above still hold provided that $\phi \in L^2_0(\partial T)$. In \eqref{eq:rellich1} and \eqref{eq:rellich1-1}, if we subtract on both sides the twice of the left hand side, and then take negative signs, we obtain:
\begin{equation}
\label{eq:rellich2}
\begin{aligned}
&\int_{\partial T} \langle \gamma,N\rangle \left((\lambda+\mu)(\div u)^2 + \mu|\nabla u|^2\right)\big\rvert_\pm = 2\int_{\partial T} \langle \gamma,N \rangle \left[(\lambda+\mu)(\mathrm{div}_t\,u)^2 + \mu|\nabla_t u|^2\right]\\
&\quad - 2\int_{\partial T} \mu(\gamma_\parallel \cdot \nabla u)\cdot \frac{\partial u}{\partial N}\Big\rvert_\pm + (\lambda+\mu)(N\cdot (\gamma_\parallel\cdot \nabla) u)(N\cdot \frac{\partial u}{\partial N})\Big\rvert_\pm\\
&\quad \pm \int_{T_\pm} (\nabla \cdot \gamma) \left[(\lambda+\mu)(\div u)^2 + \mu|\nabla u|^2\right] \mp 2\int_{T_\pm} (\nabla u\nabla \gamma):\left[(\lambda+\mu)(\div u) I + \mu \nabla u\right].
\end{aligned}
\end{equation}
Here, we used the identity:
\begin{equation*}
\nabla_t u = (I-N\otimes N)\nabla u, \qquad \mathrm{div}_t\, u = \tr(\nabla_t u).
\end{equation*}
They are, respectively, the tangential gradient of $u$ and the tangential divergence of $u$. From the trace formula \eqref{eq:traceu}, we verify that those terms together with $\gamma_\parallel\cdot \nabla u$ are continuous across $\partial T$, for $u = \cS_T[\phi]$. 
The main step to derive the formula above is to compute
\begin{equation}
\label{eq:rellichdiff}
\langle \gamma, \nabla u\rangle \cdot \frac{\partial u}{\partial \nu} - \langle \gamma, N\rangle \left[(\lambda+\mu)(\div u)^2 + \mu|\nabla u|^2\right].
\end{equation}
We use the pointwise decomposition
\begin{equation*}
\gamma = \langle \gamma,N\rangle N + \gamma_\parallel, \qquad \gamma_\parallel \in X(\partial T).
\end{equation*}
Here $X(\partial T)$ is the tangent space of $\partial T$. Then the term in \eqref{eq:rellichdiff} is hence computed as
\begin{equation*}
\begin{aligned}
&\mu\left[-\langle \gamma,N\rangle |\nabla_t u|^2 + (\gamma_\parallel\cdot \nabla u)\cdot \frac{\partial u}{\partial N}\right] + (\lambda+\mu)(\div u)\left[(\gamma_\parallel\cdot \nabla u)\cdot N - \langle \gamma,N\rangle (\mathrm{div}_t\, u)\right]\\
=&-\langle \gamma,N\rangle \left[\mu|\nabla_t u|^2 + (\lambda+\mu)(\mathrm{div}_t\,u)^2\right] + \mu(\gamma_\parallel\cdot \nabla u) \cdot \frac{\partial u}{\partial N} + (\lambda+\mu) N\cdot (\gamma_\parallel\cdot \nabla u) N\cdot \frac{\partial u}{\partial N}.
\end{aligned}
\end{equation*}

The Rellich's identities \eqref{eq:rellich2} allow us to prove the following key results.
\begin{lemma}
\label{lem:tractionbdd}
Let $d\ge 3$, let $T\subseteq \R^d$ be an open bounded set satisfying (A1) and (A2). Then there exists $C>0$, and for all $\phi \in L^2(\partial T)$, we have
\begin{equation}
\label{eq:keyineq-1}
\|(-\frac12 I + \cK^*)[\phi]\|_{L^2(\partial T)} \le C\left\{\|(\frac12 I + \cK^*)[\phi]\|_{L^2(\partial T)} + \left|\int_{\partial T} \cS_T[\phi]\right|\right\},
\end{equation}
and
\begin{equation}
\label{eq:keyineq-2}
\|(\frac12 I + \cK^*)[\phi]\|_{L^2(\partial T)} \le C\left\{\|(-\frac12 I + \cK^*)[\phi]\|_{L^2(\partial T)} + \left|\int_{\partial T} \cS_T[\phi]\right|\right\}.
\end{equation}
Moreover, for $d = 2$, the above inequalities remain valid if $\phi \in L^2_0(\partial T)$ in addition.
\end{lemma}

\begin{proof}
We only establish \eqref{eq:keyineq-1}; the other one can be proved similarly. Let $u = \cS_T[\phi]$ in $T$ and in $T_+$. By the trace formula and the definition in \eqref{eq:conormal}, we have
\begin{equation}
\label{eq:KKbdd-1}
\|(-\frac12 I + \cK^*)[\phi]\|^2_{L^2(\partial T)} = \left\|\frac{\partial u}{\partial \nu}\Big\rvert_-\right\|^2_{L^2(\partial T)} \le C\int_{\partial T} \langle \gamma,N\rangle[(\lambda+\mu)(\div u)^2 + \mu|\nabla u|^2]\big\rvert_-.
\end{equation}

\smallskip

\emph{Step 1}: Using the Rellich's identity \eqref{eq:rellich2}, we can deduce
\begin{equation}
\label{eq:KKbdd-2}
\left\|\frac{\partial u}{\partial \nu}\Big\rvert_-\right\|^2_{L^2(\partial T)}  \le C\left\{ \int_{\partial T} |\nabla_t u|^2 + \int_T (\lambda+\mu)(\div u)^2 + \mu|\nabla u|^2\right\}.
\end{equation}
Let us explain how this is done by considering a couple of typical terms on the right hand side of \eqref{eq:rellich2}. Take the second integral there for example; we can choose $c > 0$ sufficiently small so that
\begin{equation*}
 \left|\int_{\partial T} \mu(\gamma_\parallel \cdot \nabla u)\cdot \frac{\partial u}{\partial N}\right| = \left|\int_{\partial T}\mu(\gamma_\parallel \cdot \nabla_t u)\cdot \frac{\partial u}{\partial N}\right| \le c\int_{\partial T} \mu\left|\frac{\partial u}{\partial N}\right|^2 + \frac{1}{4c} \mu \|\gamma\|_{L^\infty}\|\nabla_t u\|^2_{\partial T}.
\end{equation*}
The goes to \eqref{eq:KKbdd-2} after the integral term for $\frac{\partial u}{\partial N}$ is swallowed. Let us also consider the last integral on the right hand side of \eqref{eq:rellich2}. By H\"older inequality and Young's inequality, we can choose $c > 0$ sufficiently small so that
\begin{equation*}
\begin{aligned}
&\left|\int_{T} (\nabla u\nabla \gamma):\left[(\lambda+\mu)(\div u) I + \mu \nabla u\right]\right| \le C\|\nabla \gamma\|_{L^\infty} \|\nabla u\|_{L^2} \left(\int_T (\lambda+\mu)(\div u)^2 + \mu|\nabla u|^2\right)^\frac12\\
\le &C\|\nabla \gamma\|_{L^\infty} \left(\int_T (\lambda+\mu)(\div u)^2 + \mu|\nabla u|^2\right).
\end{aligned}
\end{equation*}
This is then controlled by \eqref{eq:KKbdd-2}.

\smallskip

Next, to control \eqref{eq:KKbdd-2}, we observe that
\begin{equation*}
\int_T (\lambda+\mu)(\div u)^2 + \mu|\nabla u|^2 = \int_{\partial T} u\cdot \frac{\partial u}{\partial \nu}\Big\rvert_- = \int_{\partial T} (u - \langle u\rangle_{\partial T})\cdot \frac{\partial u}{\partial \nu}\Big\rvert_-.
\end{equation*}
Apply H\"older inequality, Poincar\'e inequality on $\partial T$, and Young's inequality, we deduce that
\begin{equation*}
\int_T (\lambda+\mu)(\div u)^2 + \mu|\nabla u|^2 \le c\left\|\frac{\partial u}{\partial \nu}\big\rvert_-\right\|^2_{L^2(\partial T)} + C\|\nabla_t u\|^2_{L^2(\partial T)}.
\end{equation*}
Using this estimate in \eqref{eq:KKbdd-2}, we get 
\begin{equation*}
\left\|\frac{\partial u}{\partial \nu}\Big\rvert_-\right\|_{L^2(\partial T)} \le C\|\nabla_t u_-\|_{L^2(\partial T)}.
\end{equation*}

\smallskip

\emph{Step 2}: We control $\|\nabla_t u\|_{L^2(\partial T)}$ by $\|\frac{\partial u}{\partial \nu}\big\rvert_+\|_{L^2}$. By continuity of tangential derivative of $\cS_T$, 
\begin{equation*}
 \|(\nabla_t u) \rvert_-\|^2_{L^2(\partial T)} = \|(\nabla_t u) \rvert_+\|^2_{L^2(\partial T)} \le \|(\nabla u) \rvert_+\|^2_{L^2(\partial T)}.
 \end{equation*}
Using the Rellich formula \eqref{eq:rellich1-1} and the same type of arguments in the previous step, we have
\begin{equation}\label{eq:KKbdd-3}
\|(\nabla u) \rvert_+\|^2_{L^2(\partial T)} \le C\left\{\left\|\frac{\partial u}{\partial \nu}\Big\rvert_+\right\|^2_{L^2(\partial T)} + \int_{T_+} (\lambda + \mu)(\div u)^2 + \mu|\nabla u|^2\right\}. 
\end{equation}
In view of \eqref{eq:cSipT}, we have the following identity
\begin{equation*}
\int_{T_+} (\lambda + \mu)(\div u)^2 + \mu|\nabla u|^2 = -\int_{\partial T}u\cdot \frac{\partial u}{\partial \nu}\Big\rvert_+ = -\int_{\partial T}(u-\langle u\rangle)\cdot \frac{\partial u}{\partial \nu}\Big\rvert_+ - \langle u\rangle \int_{\partial T} \frac{\partial u}{\partial \nu}\Big\rvert_+.
\end{equation*}
Note also, for $d=2$ we need $\phi \in L^2_0$ to apply the Green's identity. We now apply the Poincar\'e inequality on $\partial T$ to get
\begin{equation*}
\left|\int_{\partial T}u\cdot \frac{\partial u}{\partial \nu}\Big\rvert_+\right| \le c\|(\nabla_t u)\rvert_+\|^2_{L^2(\partial T)} + |\langle u\rangle_{\partial T}|^2 + \frac{C}{c}\left\|\frac{\partial u}{\partial \nu}\Big\rvert_+ \right\|^2.
\end{equation*}
Using this in \eqref{eq:KKbdd-3} yields
\begin{equation*}
\|\nabla_t u\|_{L^2(\partial T)} \le C\left\{ \left\|\frac{\partial u}{\partial \nu}\big\rvert_+ \right\| + \left|\int_{\partial T} u \right|\right\}.
\end{equation*}
Combine this with the conclusion of Step 1; we complete the proof of \eqref{eq:keyineq-1}.
\end{proof}

%%%%%%%
\subsection{Proof of Lemma \ref{lem:key}}
\label{sec:proofkey}

In this section, without further specifications, the operators $\pm \frac12 I + \cK_T$ and $\pm\frac12 I + \cK^*_T$ are viewed as bounded linear transformations on $L^2(\partial T)$. In addition, assumptions in (A1) and (A2) about $T$ are always invoked. We also denoted by $\mathbb{V}_0$ the space of constant fields in $\partial T$, and view $e_j$, $j=1,\dots,d$, as a basis for $\mathbb{V}_0$.

\begin{lemma}
\label{lem:kerK}
The inclusion $\mathbb{V}_0 \subseteq \ker(-\frac12 I + \cK_T)$ holds.
\end{lemma}
\begin{proof}
We need to check $\cK_T[e_j](x) = \frac12 e_j$ for all $x \in \partial T$ and for each $j = 1,\dots, d$. This is done by using the Green's identity \eqref{eq:GreenI} with $u = \Gamma_k$ and $v = e_j$ in $T\setminus B_\delta(x)$, compute the resulted boundary integral on $T \cap \partial B_\delta(x)$, and compute the limit of this integral as $\delta \to 0$. This is standard and the details are hence omitted.
\end{proof}
\begin{lemma}\label{lem:cKsL20}
The range of $-\frac12 I + \cK_T^*$ is contained in $L^2_0(\partial T)$ and is closed. Moreover, this operator restricted to $L^2_0$ is injective.% Moreover, this operator restricted to $L^2_0$ is injective and has closed range.
% we have the following facts.
% \begin{itemize}
% \item[(1)] The restricted operator $-\frac12 I + \cK_T^* : L^2_0 \to L^2_0$ is invertible.
% \item[(2)] $\ran(-\frac12 I + \cK^*_T) = L^2_0(\partial T)$. 
% \end{itemize}
\end{lemma}

\begin{proof} \emph{Step 1}: We check that $\ran(-\frac12 I + \cK^*_T) \subseteq L^2_0$. This is true because, for each $\ell = 1,\dots, d$, and for any $\phi \in L^2(\partial T)$ and in view of the previous lemma, we have
\begin{equation*}
\int_{\partial T} e_\ell \cdot (-\frac12 I + \cK^*_T)[\phi] = \int_{\partial T} (-\frac12 I + \cK_T)[e_\ell] \cdot \phi = 0.
\end{equation*}

\smallskip

\emph{Step 2}: We show $\ker(-\frac12 I + \cK^*_T)\cap L^2_0 = \{0\}$; in other words, $-\frac12 I + \cK^*_T$ is injective from $L^2_0$ to $L^2_0$. Suppose $\phi$ is an element in this intersection. Let $u = \cS_T[\phi]$. Then we have
\begin{equation*}
\cL^{\lambda,\mu} [u]  = 0 \text{ in } T_\pm, \qquad \frac{\partial u}{\partial \nu}\Big\rvert_- = 0, \qquad \frac{\partial u}{\partial \nu}\Big\rvert_+ = \phi, \qquad \int_{\partial T} \phi = 0.
\end{equation*}
By the Green's identity and by the continuity of $u$ across $\partial T$, we first get $u$ is a constant in $\ol T$. Since $\phi \in L^2_0$, the Green's identity \eqref{eq:greenTp} holds for all $d\ge 2$. The left hand side of \eqref{eq:greenTp} vanishes because of the observations above. Hence, $u$ is a constant over $\R^d$. The conormal of $u$ computed from $T_+$ is then zero, i.e. $\phi = 0$.

\smallskip

\emph{Step 3}: Since $L^2_0$ has finite codimension $d$, we confirm $\ran(-\frac12 I + \cK^*_T)$ is closed by showing that the restricted operator $-\frac12 I + \cK^*_T : L^2_0 \to L^2_0$ has closed range. 

Now suppose $\{g_j\} \subseteq L^2_0(\partial T)$ that satisfies $g_j \in \ran(-\frac12 I + \cK^*_T)$ and $g_j \to g$ strongly in $L^2_0$. We need to check that $g \in \ran(-\frac12 I + \cK^*_T)$. By assumption, we can find $\{h_j\} \subset L^2_0(\partial T)$ such that
\begin{equation*}
(-\frac12 I + \cK^*_T)[h_j] = g_j, \qquad \text{in } \partial T.
\end{equation*}
If the set $\{h_j\}$ is bounded, then there exists a subsequence still denoted by $\{h_j\}$, and $h_j \to h$ weakly in $L^2_0(\partial T)$. For any $\phi \in L^2(\partial T)$, we have
\begin{equation}
\label{eq:lem:Kstar-2}
\begin{aligned}
\langle g,  \phi\rangle_{L^2(\partial T)} &= \lim_{j\to \infty} \langle g_j , \varphi \rangle_{L^2(\partial T)} = \lim_{j\to \infty} \langle h_j, (-\frac12 I + \cK_T)[\phi] \rangle_{L^2(\partial T)} \\
&= \langle h,  (-\frac12 I + \cK_T)[\phi]\rangle_{L^2(\partial T)}  =  \langle(-\frac12 I + \cK_T^*)[h], \phi\rangle_{L^2(\partial T)}. 
\end{aligned}
\end{equation}
Since $\phi$ is arbitrary, we must have $g = (-\frac12 I + \cK^*_T)[h]$. The claim of this step follows in this case.

If $\{h_j\}$ is unbounded, we may assume (by extracting a subsequence if necessary) that $\|h_j\| \to \infty$. Then define $\tilde h_j = h_j/\|h_j\| \in L^2_0$; they satisfy
\begin{equation}
\label{eq:lem:Kstar-3}
\|\tilde h_j\|_{L^2(\partial T)} = 1, \quad\text{and}\quad (-\frac12 I + \cK^*_T)[\tilde h_j] = \frac{g_j}{\|h_j\|} \to 0 \text{ as } j\to \infty.
\end{equation}
We may assume that $\tilde h_j$ converges weakly to some $\tilde h \in L^2_0(\partial T)$. Very similar to \eqref{eq:lem:Kstar-2}, we can conclude that $(-\frac12 I + \cK^*_T)[\tilde h] = 0$. By the injectivity established in Step 2, we confirm that $\tilde h = 0$, and $\tilde h_j$ converges weakly in $L^2_0$ to $0$. Moreover, we abuse notations and denote the trace of $\cS_T[\phi]$ on $\partial T$ still by $\cS_T[\phi]$. It is clear that, from the properties of \eqref{eq:Gamma}, $\cS_T$ is a compact linear transform on $L^2(\partial T)$, and $\cS_T$ is self-adjoint. In particular, we have
\begin{equation*}
\int_{\partial T} \cS_T[\tilde h_j]\cdot e_k = \langle \tilde h_j, \cS_T[e_k]\rangle_{L^2(\partial T)} \to 0, \qquad \text{as $j \to \infty$}. 
\end{equation*} 
Now we use Lemma \ref{lem:tractionbdd} (this can be done for $d\ge 2$, as $\tilde h_j \in L^2_0$), by the above convergence and by the strong convergence in \eqref{eq:lem:Kstar-3}, we deduce that
\begin{equation*}
(\frac12 I + \cK^*_T)[\tilde h_j] \to 0 \quad \text{strongly in $L^2$ as $j\to \infty$}.
\end{equation*}
Combine this with \eqref{eq:lem:Kstar-3} again, we have shown that $\tilde h_j$ converges strongly to $0$ in $L^2$. It should follow that $\|\tilde h_j\| \to 0$, but this is a contradiction with \eqref{eq:lem:Kstar-3}. Hence, $\{h_j\}$ cannot be unbounded, and the conclusion of this step holds.
\end{proof}

\medskip

\begin{proof}[Proof of Lemma \ref{lem:key}] The closedness of $\ran(-\frac12 I + \cK^*_T)$ is established in Lemma \ref{lem:cKsL20}, and by duality, $\ran(-\frac12 I + \cK_T)$ is also closed. We prove rest of the conclusions in Lemma \ref{lem:key} in several steps.

\emph{Step 1:} We show that $\ker(-\frac12 I + \cK_T)$ and $\ker(-\frac12 I + \cK^*_T)$ both have dimension $d$, and characterize the first space.

Since $-\frac12 I + \cK^*_T : L^2_0 \to L^2_0$, and since $L^2_0(\partial T)$ has codimension $d$, we deduce that $\dim \ker(-\frac12 I + \cK^*_T) \le d$. On the other hand, Lemma \ref{lem:kerK} shows $\dim \ker(-\frac12 I + \cK_T) \ge d$. Now that both $-\frac12 I + \cK^*_T$ and $-\frac12 I + \cK_T$ have closed ranges, and their difference forms a compact operator (see the discussions below formula \eqref{eq:Kdifference}), we conclude, using Lemma \ref{lem:semifred}, that
\begin{equation*}
\dim \ker(-\frac12 I + \cK_T) = \dim \ker(-\frac12 I + \cK^*_T).
\end{equation*}
Those dimensions then must equal to $d$. In particular, we have $\ker(-\frac12 I + \cK_T) = \mathbb{V}_0$. As a byproduct, we also have $\ran(-\frac12 I + \cK^*_T) = \mathbb{V}_0^\perp = L^2_0(\partial T)$, and $-\frac12 I + \cK^*_T$, when restricted to $L^2_0(\partial T)$, is a bijection.

\emph{Step 3}: We establish the direct-sum decomposition (not orthogonal in general)
\begin{equation}
\label{eq:decstar}
L^2(\partial T) = \ran(-\frac12 I + \cK^*_T) \oplus \ker(-\frac12 I + \cK^*_T).
\end{equation}
Since the codimension of the first space matches the dimension of the second space, it remains to show their intersection contains only $\{0\}$. This is essentially proved by Step 2 in the proof of Lemma \ref{lem:cKsL20}.

\smallskip

\emph{Step 4}: We establish the direct-sum decomposition in item (2) of Lemma \ref{lem:key}, which, again, is not orthogonal in general. This follows directly from the decomposition in the previous step, and from the orthogonal decomposition
\begin{equation*}
L^2(\partial T) = \ran(-\frac12 I + \cK_T)\oplus \ker(-\frac12 I + \cK^*_T) = \ran(-\frac12 I + \cK^*_T)\oplus \ker(-\frac12 I + \cK_T).
\end{equation*}
This completes the proof.
\end{proof}

The following fact is a direct consequence of the proofs above.
\begin{corollary}
The operator $-\frac12 I + \cK_T : L^2_0 \to \ran(-\frac12 I + \cK_T)$ is invertible.
\end{corollary}

Our next goal is to derive a formula for the decomposition of $L^2(\partial T)$ stated in Lemma \ref{lem:key}. 

We have seen $\ker(-\frac12 I + \cK_T)$ and $\ker(-\frac12 I + \cK^*_T)$ both have dimension $d$. Following an argument in \cite[Theorem 2.26]{AmmKan} which treated layer potentials for the Laplace equation, we consider a mapping between $\ker(-\frac12 I + \cK^*_T)\times \R^d$ and $\ker(-\frac12 I + \cK_T) \times \R^d$. Both of them are product Hilbert space of dimension $2d$, and both are equipped with the standard inner product. The mapping is:
\begin{equation*}
\begin{aligned}
\mathcal{A}_T \quad \,:\,\quad \ker(-\frac12 I + \cK^*_T) \times \R^d \quad &\to \quad \ker(-\frac12 I + \cK_T) \times \R^d,\\
(\varphi, a) \quad &\mapsto \quad (\cS_T[\varphi]+a, \int_{\partial T} \varphi).
\end{aligned}
\end{equation*}
Here, the notation $\cS_T$ is abused to denote the trace on $\partial T$ of the single-layer potential. The mapping is well defined because, if $\phi \in \ker(-\frac12 I + \cK^*_T)$, then by the Green's identity \eqref{eq:greenTm}, $\cS_T[\phi]$ must be a constant in $\ol T$.

We claim that $\mathcal{A}_T$ is a bijection. It suffices to check the injectivity. Suppose $(\varphi,a)$ is such that $\varphi \in \ker(-\frac12 I + \cK^*_T)$ and $a \in \R^d$, and
\begin{equation*}
\int_{\partial T} \varphi = 0, \qquad \cS_T[\varphi] + a = 0.
\end{equation*}
By the decomposition \eqref{eq:decstar}, we conclude that $\varphi = 0$, and then $a = 0$. This proves the claim.

\begin{remark}
A very similar argument actually shows that, for $d\ge 3$, the mapping

\begin{equation*}
\begin{aligned}
\cS_T \quad:\quad \ker(-\frac12 I + \cK^*_T) \; &\to \; \mathbb{V}_0 = \ker(-\frac12 I + \cK_T)\\
\phi \; &\mapsto \; \cS_T[\phi]\rvert_{\partial T}.
\end{aligned}
\end{equation*}
is also a bijection. This is not true, in general, for $d = 2$. We will come back to this point.
\end{remark}

Now, for each $j=1,\dots,d$, consider the vector $(0,e_j)$ which is in the range of $\mathcal{A}_T$, we can find a unique pair $(\phi^*_j,a^*_j)$, with $\phi^*_j \in \ker(-\frac12 I + \cK^*_T)$ and $a^*_j \in \R^d$, as the preimage of $(0,e_j)$, i.e.
\begin{equation}
\label{eq:phistar}
\cS_T[\phi^*_j] = -a^*_j \; \text{ on } \; \ol T, \quad\text{and}\quad \int_{\partial T} \phi^*_j = e_j.
\end{equation}
Clearly, $\{\phi^*_j\}$ form a basis for $\ker(-\frac12 I + \cK^*_T)$.  Let $A_T$ be the matrix with $a^*_j$'s as columns, i.e. $A_T$ is defined by \eqref{eq:ATdef}. It has the following nice properties.

\begin{proposition}
\label{prop:AT}
For $d \ge 2$, the matrix $A_T$ is symmetric. For $d\ge 3$, $A_T$ is positive definite.
\end{proposition}

\begin{proof} We can write the component of $A_T$ as
\begin{equation*}
(a^*_j)^i = -e_i \cdot \cS_T[\phi^*_j] =  -(\int_{\partial T} \phi^*_i) \cdot \cS_T[\phi^*_j] = -\langle \phi^*_i, \cS_T[\phi^*_j]\rangle_{L^2(\partial T)}.
\end{equation*}
Using the fact that $\cS_T$ is self-adjoint, we can rewrite the right hand side as $-\langle \cS_T[\phi^*_i],\phi^*_j\rangle_{L^2(\partial T)}$, which is, according to the formula above, $(a^*_i)^j$. Hence, $A_T$ is symmetric.

Now we impose the condition $d\ge 3$. To check that $A_T$ is positive definite, consider any vector $c = (c^i) \in \R^d$ and we compute that
\begin{equation*}
(A_T c) \cdot c = -\int_{\partial T} \phi \cdot \cS_T[\phi],
\end{equation*} 
where $\phi = c^i \phi^*_i$ which belongs to $\ker(-\frac12 I + \cK^*_T)$. Let $u = \cS_T[\phi]$ in $\R^d$, we can recast the above identity as
\begin{equation*}
(A_T c) \cdot c = -\int_{\partial T} \frac{\partial u}{\partial \nu}\Big\vert_+ \cdot u.
\end{equation*}
In $d \ge 3$, we can apply the Green's identity \eqref{eq:greenTp} and conclude that
\begin{equation*}
(A_T c) \cdot c = \int_{\R^d\setminus \ol T} \mu |\nabla u|^2 + (\lambda + \mu)(\div u)^2.
\end{equation*}
The right hand side is non-negative, and it vanishes if and only if $u = \cS_T[\phi]$ is a constant on $T_+$, which would imply $\phi = c^i \phi^*_i = 0$, and finally $c = 0$. This shows $A_T$ is positive definite for $d\ge 3$.
\end{proof}

\begin{remark}\label{rem:AT2d} 
For $d = 2$, the matrix $A_T$ can be degenerate. In fact, there is an abnormal rescaling for $\cS_T$, which is due to the logarithmic term in $\Gamma_k$. Indeed, for $d = 2$, we note from \eqref{eq:Gamma} that, for any $r > 0$,
\begin{equation*}
\Gamma^j_k(\frac{x}{r}) = \Gamma^j_k(x) - \frac{c_1}{2\pi}(\log r) \delta_{jk}.
\end{equation*}
We then have
\begin{equation*}
\begin{aligned}
(\cS_T[\phi])(x) =& \int_{\partial T} \Gamma^j_k(\frac{x-y}{r}) \phi^j(y) dy + \frac{c_1}{2\pi}(\log r)\int_{\partial T} \phi\\
=& r\int_{\partial (\frac1r T)} \Gamma^j_k(\frac{x}{r} - z)\phi^j(rz) dz + \frac{c_1}{2\pi}(\log r)\int_{\partial T} \phi\\
=& r\cS_{\frac1r T} [\phi(r\cdot)](\frac{x}{r}) + \frac{c_1}{2\pi}(\log r)\int_{\partial T} \phi.
\end{aligned}
\end{equation*}
Consider the $\phi^*_j$'s in \eqref{eq:phistar}, and let $\phi^*_{j,r} \in L^2(\partial(\frac1r T))$ be the rescaled function
\begin{equation*}
\phi^*_{j,r}(z) = r \phi^*_j(r z), \qquad z \in \frac1r T.
\end{equation*}
Then we can check that
\begin{equation*}
\int_{\partial (\frac1r T)} \phi^*_{j,r}(z) dz = \int_{\partial T} \phi^*_j(y) dy = e_j,
\end{equation*}
and meanwhile, due to the homogeneity (of degree $-1$) of the integral kernel $K^*_{ik}$, we also have
\begin{equation*}
(-\frac12 I + \cK^*_{\frac{1}{r} T})[\phi^*_{j,r}](z) = r(-\frac12 I + \cK^*_T)[\phi](rz), \qquad z \in \frac{1}{r} \ol T.
\end{equation*}
In particular, $\phi^*_{j,r}$'s belong to $\ker(-\frac12 I + \cK^*_{\frac1r T})$. Finally, from the rescaling formula of $\cS_T$, we found that
\begin{equation*}
A_{rT} = A_T + \frac{c_1}{2\pi}(\log r) I, \qquad r > 0.
\end{equation*}
From this relation, we can see that, given a shape $T$, there always exist one or two $r > 0$ such that $A_{rT}$ can be degenerate, and there are at most two such $r$.

As a consequence, for $d=2$ and when the homogenization of \eqref{eq:hetlame} is considered for the dilute case, we can always assume $\det A_T \ne 0$. Indeed, if this fails, we can replace it by $r_0T$ for $r_0$ slightly less than one so that $\det A_{r_0T} \ne 0$. Because we are interested in $\eps \to 0$ only, the geometric set-up of the homogenization problem does not change once we replace $\eta$ by $\eta/r_0$.
\end{remark}

Finally, the proof above provides a formula for the decomposition.

\begin{lemma}\label{lem:L2decom}
Suppose $d\ge 2$, $T \subseteq \R^d$ is an open bounded set satisfying (A1) and (A2); for $d = 2$, we further assume \eqref{eq:AT2d}. Let $\{\phi^*_j\}$ be defined by \eqref{eq:phistar}.
Let $\Pi_0 : L^2(\partial T) \to \ker(-\frac12 I + \cK_T)$ and $\Pi_1 := I - \Pi_0$ be the projection operators to $\ker(-\frac12 I + \cK_T)$ and to $\ran(-\frac12 I + \cK_T)$. That is, for $\phi \in L^2(\partial T)$, $(\Pi_0[\phi],\Pi_1[\phi])$ be the unique pair such that
\begin{equation*}
\phi = \Pi_0[\phi] + \Pi_1[\phi], \quad \text{with} \quad \Pi_0[\phi] \in \ker(-\frac12 I + \cK_T), \; \Pi_1[\phi] \in \ran(-\frac12 I + \cK_T).
\end{equation*}
Then we have
\begin{equation*}
(\Pi_0[\phi])^k = \langle \phi^*_k,\phi\rangle.
\end{equation*}
% Then for any $\phi \in L^2(\partial T)$, the unique pair $(c,\hat\phi)$ that satisfy
% \begin{equation*}
% \phi = c + \hat\phi, \qquad c\in \R^d, \; \hat\phi \in \ran(-\frac12 I + \cK_T),
% \end{equation*}
% is given by $c = (c^j)$ and $\hat\phi = \phi - c$, with $c^j$ defined by
% \begin{equation*}
% c^j = \langle \phi^*_h,\phi\rangle = \int_{\partial T} \phi^*_j(y) \cdot \phi(y) dy.
% \end{equation*}
\end{lemma}

%%%%%%%
\subsection{Periodic layer potentials}

To solve the cell problem, we use periodic layer potentials. They are variants of the aforementioned layer potentials adapted for Lam\'e systems in the torus $\bT^d$, or in the rescaled torus $\eta^{-1}\bT^d$. In this subsection, assumptions (A1), (A2) and (A3) are all invoked.

We start with the unit torus, and consider the fundamental solution $G_k(x)$ that solves
\begin{equation}
\label{eq:pLame}
\cL^{\lambda,\mu} [G_k](x) =  (\delta_0(x) -1) e_k, \qquad \text{in } \bT^d,
\end{equation}
with the normalization condition
\begin{equation*}
\int_{\bT^d} G_k(x) = 0.
\end{equation*}
It is straightforward to check that, for each $k = 1,\dots,d$, there is a unique solution, $G_k$ is smooth in $\bT^d\setminus \{0\}$. Moreover, $G_k$ can be viewed as a ``perturbation'' of the free-space solution $\Gamma_k$, in the sense that there exists a unique $R_k(x) \in C^\infty([-\frac12,\frac12]^d)\cap C(\bT^d)$, such that
\begin{equation*}
G_k(x) = \Gamma_k(x) + R_k(x), \qquad \forall x \in \bT^d\setminus \{0\}.
\end{equation*}
In fact, derivatives of $R_k$ do not satisfy periodicity, so $R_k$ is not an element of $C^1(\bT^d)$. For rather explicit Fourier representations for $R_k$, we refer to \cite{AmmKan}.

On the rescaled torus $\eta^{-1}\bT^d$, we define the rescaled function
\begin{equation}
\label{eq:Getapert}
G^\eta_k(x) = \eta^{d-2} G_k(\eta x) = \Gamma_k(x) + \eta^{d-2} R_k(\eta x).
%\begin{cases}
%\Gamma_k(x) + \eta^{d-2} R_k(\eta x), \qquad &d\ge 3,\\
%\Gamma_k(x) + R_k(\eta x) + \frac{c_1}{2\pi}(\log \eta) e_k, \qquad &d=2.
%\end{cases}
\end{equation}
Note that for $d=2$, we abuse notations and have subtracted a constant term of the form $\frac{c_1}{2\pi}(\log \eta) e_k$ in the second equality. In view of the scaling property of the Dirac distribution, we check that $G^\eta_k$ solves the problem
\begin{equation}
\label{eq:pLameeta}
\cL^{\lambda,\mu} [G^\eta_k](x) = (\delta_0(x) -\eta^d) e_k, \qquad \text{in } \eta^{-1}\bT^d.
\end{equation}
Using those fundamental solutions, we define the periodic single-layer potential associated to $T$, for $\phi \in L^2(\partial T)$, by
\begin{equation*}
(\cS_T^\eta[\phi])^k(x) = \int_{\partial T} G^\eta_k(x;y)\cdot \phi(y) dy, \qquad x\in \eta^{-1}\bT^d\setminus \partial T,
\end{equation*}
and define the periodic double-layer potential by
\begin{equation*}
(\cD^\eta_T [\phi])^k (x) = \pv \int_{\partial T} \frac{\partial G^\eta_k(\eta(x-y))}{\partial \nu_y} \cdot \phi(y) dy.
\end{equation*}
It is important to point out that $\cL^{\lambda,\mu}[\cS^\eta_T[\phi]] = 0$ in $T$ and in $\frac1\eta \bT^d\setminus \ol T$ only for $\phi \in L^2_0(\partial T)$; on the other hand, $\cL^{\lambda,\mu}[\cS^\eta_T[\phi]] = 0$ away from $\partial T$ for all $\phi \in L^2$.

In view of the decomposition of $G^\eta_k$, we can write
\begin{equation*}
\cS_T^\eta = \cS_T + \eta^{d-2} \cS^\eta_{T,1}, \qquad \cS^\eta_{T,1}[\phi] = \int_{\partial T} R_k(\eta(x-y))\cdot \phi(y) dy.
\end{equation*}
Because $R_k(\eta(x-y))$ is uniformly bounded with respect to $\eta,x$ and $y$, the operator $\cS^\eta_{T,1}$ is uniformly bounded (in $\eta$) and compact on $L^2(\partial T)$. Moreover, because $\nabla R_k$ is uniformly bounded, $\cS^\eta_{T,1}$ can be differentiated.
%, and we check that
% \begin{equation*}
% \partial_j (\eta^{d-2}\cS_T[\phi])^k(x) = \eta^{d-1} \left(\int_{\partial T} \partial_j R^\ell_k (\eta(x-y)) \phi^\ell(y) dy\right)_{kj}.
% \end{equation*}
We then have the following trace formulas
\begin{equation*}
\frac{\partial \cS^\eta_T[\phi]}{\partial \nu}\Big\rvert_\pm (x) = \left(\pm\frac12 I + \cK^{\eta,*}_T\right)[\phi], \qquad x\in \partial T,
\end{equation*}
where $\cK^{\eta,*}_{T} = \cK^*_T + \eta^{d-1}\cK^{\eta,*}_{T,1}$ and
\begin{equation*}
\cK^{\eta,*}_{T,1}[\phi] = \int_{\partial T} (\lambda+\mu)(\nabla\cdot R_k)(\eta(x-y))\langle N_x,\phi(y)\rangle + \mu (N_x\cdot \nabla R_k(\eta(x-y)))\cdot \phi(y) \,dy.
\end{equation*}
In particular, $\cK^{\eta,*}_{T,1}$ is a compact operator on $L^2(\partial T)$ that is uniformly bounded in $\eta$.

\smallskip

Similarly, for the double-layer potential, we also have
\begin{equation*}
\cD^\eta_T = \cD_T + \eta^{d-1} \cD^\eta_{T,1},
\end{equation*}
where the perturbation operator $\cD^\eta_{T,1}$ is defined by
\begin{equation*}
\cD^\eta_{T,1}[\phi](x) = -\int_{\partial T} \left[(\lambda+\mu)(\nabla\cdot R_k)(\eta(x-y))N_y + (\mu N_y\cdot \nabla R_k)(\eta(x-y))\right]\cdot \phi(y) dy.
\end{equation*}
The trace formulas are
\begin{equation*}
\cD^\eta_T[\phi]\Big\rvert_\pm (x) = \left(\pm\frac12 I + \cK^{\eta}_T\right)[\phi], \qquad x\in \partial T,
\end{equation*}
where $\cK^{\eta}_{T} = \cK_T + \eta^{d-1}\cK^{\eta}_{T,1}$ and $\cK^\eta_{T,1}$ is simply the restriction of $\cD^\eta_{T,1}$ on $\partial T$. Again, because $\nabla R_k$ is uniformly bounded in $[-\frac12,\frac12]^d$, the integral kernel above is bounded and the resulted operator is compact in $L^2(\partial T)$ and its operator norm is uniformly bounded.

The trace formulas for $\cD^\eta_T$ can be used to solve the Dirichlet boundary value problems, namely \eqref{eq:chieta}. The following facts will be useful.

\begin{theorem}
\label{thm:perK}
For the operators $-\frac12 I + \cK^\eta_T$ and $-\frac12 I + \cK^{\eta,*}_T$, the following holds.
\begin{itemize}
\item[(1)] For each $\ell = 1,\dots,d$, $(-\frac12 I + \cK^\eta_T)[e_\ell] = -\eta^d|T| e_\ell$.
\item[(2)] The operators $-\frac12 I + \cK^{\eta,*}_T$ and $-\frac12 I + \cK^\eta_T$ are bijections in $L^2(\partial T)$. 
\end{itemize} 
\end{theorem} 

\begin{proof}
Item (1) is a direct computation and follows from the Green's identity in the domain $\eta^{-1}\bT^d\setminus \ol T$. To be more precise, note that $e_\ell$ as a function solves the homogeneous Lam\'e system in $\eta^{-1}\bT^d$; it follows that, for $x \in T$,
\begin{equation*}
\int_{\partial T} \frac{\partial G^\eta_k}{\partial \nu_y}(x;y) \cdot e_\ell = \int_T e_\ell \cdot \cL^{\lambda,\mu} [G^\eta_k(x-\cdot)]  = \int_{T} (\delta_x(y)-\eta^d) e_k\cdot e_\ell = (1-\eta^d)|T| \delta_{j\ell}.
\end{equation*}
By the trace formula, we get
\begin{equation*}
(\frac12 I + \cK^\eta_T)[e_\ell] = (1-\eta^d)|T|e_\ell, \qquad (-\frac12 I + \cK^\eta_T)[e_\ell] = -\eta^d|T| e_\ell.
\end{equation*}
In particular, for any $\eta > 0$, non-zero elements in $\ker(-\frac12 I + \cK_T)$ is no longer in $\ker(-\frac12 I + \cK^\eta_T)$.

\smallskip

Suppose $\phi \in \ker(-\frac12 I + \cK^{\eta,*}_T)$, then from item (1) it follows that $\phi \in L^2_0(\partial T)$, and hence $\cS^\eta_T[\phi]$ solves the homogeneous Lam\'e system in $\eta^{-1}\bT^d\setminus \partial T$. Green's identity then shows that $\cS^\eta_T[\phi] = 0$ in $\eta^{-1}\bT^d$, and it follows that $\ker(-\frac12 I + \cK^{\eta,*}_T) = \{0\}$. On the other hand, in view of the perturbative relations and the compactness of $\cK^{\eta}_{T,1}$ and $\cK^{\eta,*}_{T,1}$, the ranges of $-\frac12 I + \cK^\eta_T$ and $-\frac12 I + \cK^{\eta,*}_T$ are still closed. Then Lemma \ref{lem:semifred} shows that $\ker(-\frac12 I + \cK^\eta_T) = \{0\}$, and that those operators are bijections on $\partial T$.
\end{proof}

%%%%%%%%%%%
%%%%%%%%%%%
\section{Asymptotic analysis for the rescaled cell problem}
\label{sec:cell}

As discussed in the Introduction, to prove homogenization results using the standard oscillating test function arguments, we need solve the rescaled cell-problem \eqref{eq:chieta}, which is imposed on $\frac1\eta \bT^d$. The existence and uniqueness of its solution $\chi^\eta_k$ can be obtained from the standard elliptic theory. Take the inner product with $\chi^\eta_k$ on both sides of \eqref{eq:chieta} and integrate by parts, we get
\begin{equation*}
\int_{\eta^{-1}\bT^d\setminus \ol T} \mu |\nabla \chi^\eta_k|^2 + (\lambda+\mu)(\div \chi^\eta_k)^2 = \eta^d \int_{\bT^d\setminus \ol T} e_k \cdot \chi^\eta_k.
\end{equation*}
Using the Poincar\'e inequality \eqref{eq:poincare}, we get
\begin{equation}
\label{eq:chigradbdd}
\|\nabla \chi^\eta_k\|_{L^2(\eta^{-1}\bT^d\setminus \ol T)} \le \begin{cases}
C, &\qquad d\ge 3,\\
C|\log\eta|^{\frac12}, &\qquad d=2.
\end{cases}
\end{equation}

To make the oscillation structure of the domain coincide with that of $D^\eps$, we define the further rescaled function
\begin{equation}
\label{eq:vepsdef}
v^\eps_k(x) = \begin{cases}
\chi^\eta_k(\frac{x}{\eps\eta}), \qquad &d\ge 3,\\
\frac{1}{|\log \eta|} \chi^\eta_k(\frac{x}{\eps\eta}), \qquad &d=2.
\end{cases}
\end{equation}
By the definition $v^\eps_k$ vanishes in the holes of $\eps\R^d_f$, and a direct computation shows that
\begin{equation}
\label{eq:vepseq}
-\cL^{\lambda,\mu}_x [v^\eps_k](x) =  \frac{1}{\sigma^2_\eps} e_k \qquad \text{in } \eps \R^d_f.
\end{equation}

We have the following result concerning the asymptotic behavior of $v^\eps_k$.

\begin{lemma}
\label{lem:vasym} Suppose the assumptions of Theorem \ref{thm:qualh} hold. Let $v^\eps_k$, $k=1,\dots,d$, be defined by \eqref{eq:vepsdef}. Then the following holds.
\begin{itemize}
\item[(1)] For all regimes of hole-cell ratios, there exists $C > 0$ depending only on $d$ and $T$ such that	 
\begin{equation}
\label{eq:vgradbdd}
\|\nabla v^\eps_k\|_{L^2(D)} \le C \sigma_\eps^{-1}.
\end{equation}
\item[(2)] In the critical setting, i.e. when $\sigma_\eps$ converges to some positive constant $\sigma_0$ as $\eps \to 0$,
\begin{equation}
\label{eq:dvweak}
\nabla v^\eps_k = (\partial_j (v^\eps_k)^j) \rightharpoonup 0 \quad \text{weakly in }\;  L^2(D).
\end{equation}
\item[(3)] For all dilute settings, i.e. when $\eta_\eps \to 0$ as $\eps \to 0$, let $M$ be defined by \eqref{eq:Mdef}. Then, for $d \ge 3$ with $p \in [1,\frac{2d}{d-2}]$, one has
\begin{equation}
\label{eq:vstrong}
 v^\eps_k \to M^{-1}e_k \quad \text{ in } L^{p}_{\rm loc}(\R^d).
\end{equation}
For $d=2$, the above holds for $p \in [1,2]$.
\end{itemize}
\end{lemma}

\begin{proof}
The gradient bound in \eqref{eq:vgradbdd} is essentially a rescaling of \eqref{eq:chigradbdd} and the proof is omitted. The proof of \eqref{eq:vstrong} is postponed to the next lemma where the results are stronger. We only establish the weak convergence \eqref{eq:dvweak} here.

We first note that in this critical hole-cell ratio setting, $\|\nabla v^\eps_k\|_{L^2}$ is uniformly bounded and, hence, it suffices to check that for all $\varphi \in C^\infty_c(D,\R)$, for all $j,\ell = 1,\dots,d$,
\begin{equation}
\label{eq:vklem-1}
\int_D (\partial_\ell v^j) \varphi \to 0, \qquad \text{as } \eps \to 0.
\end{equation}
Here and in the rest of the proof, we write $(v^\eps_k)^j$ simply as $v^j$.

Consider the $\eps$-cubes in the definition of $\eps\R^d_f$, i.e. cubes of the form $\eps (z + (-\frac12,\frac12)^d)$, $z \in \Z^d$, and label those that have non-empty intersection with $D$ by $i \in \N$. Among those cubes, let $\mathcal{I}_\eps$ denote those contained in $D$, and let $\mathcal{J}_\eps$ denote those that intersect with $\partial D$.

For a typical interior cube denoted by $Q_{\eps,i} = z_{\eps,i}+\eps (-\frac12,\frac12)^d$, where $z_{\eps,i} \in \eps \Z^d$, we compute
\begin{equation*}
\int_{Q_{\eps,i}} (\partial_\ell v^j)\varphi = \int_{Q_{\eps}} (\partial_\ell v^j) \varphi(z_{\eps,i}+y) dy = (\eps\eta)^{d-1}\int_{Q_{\frac1\eta}} (\partial_\ell \chi^\eta_k)^j(y) \varphi(z_{\eps,i} + \eps\eta y) dy.
\end{equation*}
We use Taylor expansion for $\varphi$, and check that
\begin{equation*}
\left|\varphi(z_{\eps,i} + \eps\eta y) - \varphi(z_{\eps,i})\right| \le \|\nabla \varphi\|_{L^\infty} \eps.
\end{equation*}
Since replacing $\varphi$ by $\varphi(z_{\eps,i})$ makes the integral vanish because $\partial_\ell v^j$ is periodic, we deduceb
\begin{equation*}
\left|\int_{Q_{\eps,i}} (\partial_\ell v^j)\varphi\right| \le \|\nabla \varphi\|_{L^\infty} \eps^d \eta^{d-1}\|\nabla \chi^\eta_k\|_{L^2}|Q_{\frac1\eta}|^{\frac12} \le C\eps^d \eta^{\frac{d-2}{2}}. 
\end{equation*}
The above holds for $d\ge 3$. If $d=2$, there is a further multiplicative factor $|\log\eta|^{-\frac12}$ on the right hand side, in view of the definition \eqref{eq:vepsdef} and the bound \eqref{eq:chigradbdd}. The above estimate is uniform for $i\in \mathcal{I}_\eps$. Since the number of interior cubes is of order $O(\eps^{-d})$, the overall contribution to the left hand side of \eqref{eq:vklem-1} from interior cubes vanishes in the limit.

For a typical boundary cube denoted by $Q_{\eps,i}$, $i \in \mathcal{J}_\eps$, we use H\"older inequality to get, for $d\ge 3$,
\begin{equation*}
\begin{aligned}
\left|\int_{Q_{\eps,i}} (\partial_\ell v^j)\varphi\right| &\le \|\varphi\|_{L^\infty}\|\partial_\ell v^j\|_{L^2(Q_{\eps,i})}\eps^{\frac d2}\\
&= \|\varphi\|_{L^\infty}\|\partial_\ell (\chi^\eta_k)^j\|_{L^2(Q_{\frac1\eta})}\eps^{\frac d2}(\eps\eta)^{\frac{d-2}{2}} \le C\eps^{d-1} \eta^{\frac{d-2}{2}}.
\end{aligned}
\end{equation*}
Again, for $d=2$, the right hand side is multiplied by $|\log\eta|^{-\frac12}$. Because $\mathcal{J}_\eps$ has a cardinality of order $\eps^{-d+1}$, the above estimate shows that the contribution of boundary cubes to the integral in \eqref{eq:dvweak} also vanishes in the limit. This proves \eqref{eq:dvweak}.
\end{proof}

\begin{lemma}
Under the same conditions of the previous lemma, there exists $C > 0$ depending only on $T$, $d$ and $D$, such that, for $\eps$ sufficiently small,
\begin{equation}
\label{eq:vstrbdd}
\|v^\eps_k -  M^{-1}e_k\|_{L^p(D)} \le \begin{cases}
C\eta^{\frac{d-2}{2}}, \qquad& d \ge 3 \,\text{and }\, p = \frac{2d}{d-2},\\
C|\log \eta|^{-\frac12}, \qquad& d = 2 \,\text{and }\, p = 2.
\end{cases}
\end{equation}
\end{lemma}

\begin{proof} Because $\eps$ is taken small, we can assume that \eqref{eq:AT2d} holds as remarked in \ref{rem:AT2d}. Our proof is based on an explicit representation of $\chi^\eta_k$, which is made possible by the layer potentials developed earlier. 

Compare the equations \eqref{eq:chieta} and \eqref{eq:pLameeta}, in the domain $\eta^{-1}\bT^d \setminus \ol T$, we must have 
\begin{equation*}
\chi^\eta_k(x) = G^\eta_k(x) + \Phi^\eta_k(x), \qquad x\in \eta^{-1} \bT^d\setminus \ol T,
\end{equation*}
where $\Phi^\eta_k$ is the unique solution to
\begin{equation}
\label{eq:Phietaeq}
\cL^{\lambda,\mu} [\Phi^\eta_k] = 0 \, \quad \text{ in }\, \eta^{-1}\bT^d\setminus \ol T, \qquad \Phi^\eta_k = - G^\eta_k \, \text{  on } \, \partial T.
\end{equation}
This is a Dirichlet boundary problem for the Lam\'e system on the torus $\eta^{-1}\bT^d$ and exterior to $T$. We can solve it using the double-layer potential $\cD^\eta_T$. However, to obtain necessary estimates, we first perform a decomposition of the boundary data according to Lemma \ref{lem:L2decom}. We have
\begin{equation}
\label{eq:phi-2}
-G^\eta_k = c^\eta_k + h^\eta_k,
\end{equation}
with $h^\eta_k \in \ran(-\frac12 I + \cK_T)$ and $c^\eta_k \in \R^d$. In view of the decomposition formula and the perturbation relation \eqref{eq:Getapert}, we have
\begin{equation*}
(c^\eta_k)^j = -\int_{\partial T} G^\eta_k(y)\cdot \phi^*_j(y) dy = -(\cS_T[\phi^*_j])^k(0) - \eta^{d-2}\cS^\eta_{T,1}[\phi^*_j](0) = -(a^*_k)^j  - \eta^{d-2}\cS^\eta_{T,1}[\phi^*_j](0).
\end{equation*}
In particular, the last term is a constant of order $O(\eta^{d-2})$. On the other hand, since $-\frac12 I + \cK^\eta_T$ is invertible on $L^2(\partial T)$, we can find a unique $g \in L^2(\partial T)$ such that
\begin{equation}
\label{eq:phi-1}
h^\eta_k = (-\frac12 I + \cK^\eta_T)[g] = -\eta^d|T|\langle g\rangle + (-\frac12 I + \cK_T)[g'] + \eta^{d-1}\cK^\eta_{T,1}[g'],
\end{equation}
where $\langle g\rangle := \fint_{\partial T} g $ is the mean-value of $g$ on $\partial T$, and $g' \in L^2_0(\partial T)$ is the fluctuation, and $g = g' + \langle g\rangle$.

Let $\Pi_1$ in Lemma \ref{lem:L2decom} operate on both sides of \eqref{eq:phi-1}, we get
\begin{equation*}
(-\frac12 I + \cK_T + \eta^{d-1}\Pi_1 \cK^\eta_{T,1})[g'] =  h^\eta_k.
\end{equation*}
The operator $\Pi_1\cK^\eta_{T,1}$ is compact on $L^2(\partial T)$ and the left hand side is hence a perturbation to $-\frac12 I + \cK_T$, which is invertible from $L^2_0(\partial T)$ to $\ran(-\frac12 I + \cK_T)$. We conclude that, for $\eta$ sufficiently small, the perturbed operator remains invertible and
\begin{equation*}
g' = (-\frac12 I + \cK_T + \eta^{d-1}\Pi_1 \cK^\eta_{T,1})^{-1}[h^\eta_k].
\end{equation*}
Both the inversion operator and $h^\eta_k$ can be uniformly bounded in $\eta$; we conclude that $\|g'\|_{L^2} \le C$. Finally, let the projection $\Pi_0$ operate on both sides of \eqref{eq:phi-1}, we get
\begin{equation*}
-\eta^d|T|\langle g\rangle + \eta^{d-1} \Pi_0\cK^\eta_{T,1}[g'] = 0. 
\end{equation*}
From this we deduce that $\langle g\rangle = O(\eta^{-1})$.

The the solution to the rescaled cell problem \eqref{eq:chieta} is hence represented by
\begin{equation}
\label{eq:phi-3}
\chi^\eta_k = \Gamma_k + A_T e_k +  \cD^\eta_T[g'] + O(\eta^{d-2}).
\end{equation}
The error term has an $L^\infty$ norm of order $\eta^{d-2}$, and it includes the constant error in \eqref{eq:phi-2}, the perturbation in \eqref{eq:Getapert} and the constant term in \eqref{eq:phi-1}. 

Back to the proof of \eqref{eq:vstrbdd}. %We fix a $k = 1,\dots,d$ and simply write $v^\eps_k$ as $v^\eps$. 
We decompose the integral over $D$ into integrations over $\eps$-cubes as before, and consider first the case of $d\ge 3$. Let $p = \frac{2d}{d-2}$. We compute
\begin{equation*}
\|v^\eps_k - M^{-1}e_k\|_{L^p(D)}^p \le \sum_{i \in \mathcal{I}_\eps}  \int_{Q_{\eps,i}} |v^\eps_k(z) - M^{-1}e_k|^p dz.
\end{equation*}
Here, $\mathcal{I}_\eps$ is the index set for $\eps$-cubes that has non-empty intersection with $D$. In each $\eps$-cube, we estimate the integral by
\begin{equation}
\label{eq:vav-1}
\begin{aligned}
\int_{Q_{\eps,i}} |v^\eps_k (z) - M^{-1}e_k|^p dz \le& C \int_{Q_{\eps,i}}|v^\eps - \langle v^\eps\rangle_{Q_{\eps,i}}|^p + |\langle v^\eps_k\rangle_{Q_{\eps,i}} - M^{-1}e_k|^p\\
\le& C\left( \|\nabla v^\eps_k\|_{L^2(Q_{\eps,i})}^p + \eps^d |\langle v^\eps_k\rangle_{Q_{\eps,i}} - M^{-1}e_k|^p\right).
\end{aligned}
\end{equation}
We used the Sobolev embedding $L^{2^*}(r\bT^d)\subseteq H^1(r\bT^d)$, for any $r >0$, where $r\bT^d$ is the rescaled torus; moreover, the bounding constant in the embedding inequality is scaling invariant and hence independent of $r$. The constant $C$ above hence depends only on $p$ and $d$. We have
\begin{equation}\label{eq:vav-2}
\|\nabla v^\eps_k\|_{L^2(Q_{\eps,i})}^2 = (\eps\eta)^{d-2}\|\nabla \chi^\eta_k\|_{L^2(\eta^{-1}\bT^d)}^2 \le C(\eps\eta)^{d-2}.
\end{equation}
To control the contribution of $\langle v^\eps_k\rangle_{Q_{\eps,i}} - M^{-1}e_k$, we compute and find that
\begin{equation*}
\langle v^\eps_k\rangle_{Q_{\eps,i}} - M^{-1}e_k = \langle \chi^\eta\rangle_{\eta^{-1}\bT^d} - M^{-1}e_k = \langle \chi^\eta_k\rangle_{\frac{1}{\eta}\bT^d\setminus \ol T} - M^{-1}e_k + O(\eta^d)
\end{equation*}
From \eqref{eq:phi-3}, we have
\begin{equation*}
\langle \chi^\eta_k\rangle_{\frac{1}{\eta}\bT^d\setminus \ol T} - M^{-1}e_k = \langle \Gamma_k\rangle_{\frac1\eta \bT^d\setminus \ol T} + \langle \cD^\eta_T[g']\rangle_{\frac1\eta\bT^d \setminus \ol T} + O(\eta^{d-2}).
\end{equation*}
We need to estimate the first two terms on the right hand side. For the average of $\Gamma_k$, we note that
\begin{equation*}
|\Gamma_k(x)| \le \frac{C}{|x|^{d-2}}.
\end{equation*}
As a result, 
\begin{equation*}
\left|\int_{\frac{1}{\eta}\bT^d\setminus \ol T} \Gamma_k(x) dx\right| \le \int_{\frac{1}{\eta}\bT^d} \frac{C}{|x|^{d-2}} \le C\eta^{-2}, \quad \text{and} \quad
\left|\langle \Gamma_k\rangle_{\frac1\eta \bT^d\setminus \ol T}\right| \le C \eta^{d-2}.
\end{equation*}

For the second term, we compute
\begin{equation*}
\begin{aligned}
\int_{\frac1\eta \bT^d \setminus \ol T} \cD^\eta_T[g'](x) dx &= \int_{\frac1\eta \bT^d \setminus \ol T} \int_{\partial T} \left[(\lambda+\mu)(\div_y \Gamma_k(x;y))N_y + \mu N_y\cdot \nabla_y \Gamma_k(x;y) \right] \cdot g'(y) dy dx\\
&= -\int_{\frac1\eta \bT^d \setminus \ol T} \int_{\partial T} \left[(\lambda+\mu)(\div_x \Gamma_k(x;y))N_y + \mu N_y\cdot \nabla_x \Gamma_k(x;y) \right] \cdot g'(y) dy dx\\
&= \int_{\partial T} \int_{\partial T} \left[(\lambda+\mu)(N_x\cdot \Gamma_k(x;y) N_y + \mu N_y\cdot N_x \Gamma_k(x;y) \right] \cdot g'(y) dx dy.
\end{aligned}
\end{equation*}
Using the fact
\begin{equation*}
\sup_{y \in \partial T} \int_{\partial T} |\Gamma_k(x;y)| dx \le C,
\end{equation*} 
we deduce that
\begin{equation*}
\langle \cD^\eta_T[g']\rangle_{\frac1\eta \bT^d\setminus \ol T} \le C\eta^d.
\end{equation*}
It follows that
\begin{equation}
\label{eq:vav-3}
\left|\langle v^\eps_k\rangle_{Q_{\eps,i}} - M^{-1}e_k\right| \le C\eta^{d-2}.
\end{equation}

Use all the estimates above in \eqref{eq:vav-1}, we conclude that
\begin{equation*}
\|v^\eps_k - M^{-1}e_k\|_{L^p(Q_{\eps,i})}^p \le C\eps^d\eta^d.
\end{equation*}
This estimate is uniform for all the cubes $Q_{\eps,i}$'s, and there are $O(\eps^{-d})$ many of them. We hence conclude that
\begin{equation*}
\|v^\eps_k - M^{-1}e_k\|_{L^p(Q_{\eps,i})} \le C\eta^{\frac{d}{p}} = C\eta^{\frac{d-2}{2}}.
\end{equation*}
This completes the proof for $d\ge 3$.

In the two dimensional case, we repeat the argument above but for $p = 2$. In this case, we have
\begin{equation*}
\begin{aligned}
v^\eps_k(x) =& \frac{1}{|\log \eta|} \chi^\eta(\frac{x}{\eps \eta})
= \frac1{|\log \eta|} \left[\Gamma_k(\frac{x}{\eps \eta}) + A_T e_k + \cD^\eta_T[g'] + O(1)\right]\\
=& \frac{1}{|\log \eta|} \left[\frac{c_1}{2\pi}\left(\log\left|\frac{x}{\eps}\right|\right) e_k + \frac{c_1}{2\pi} \log \frac1\eta e_k + \cD^\eta_T[g'] + O(1) \right].
\end{aligned}
\end{equation*}
In particular, we note that
\begin{equation*}
v^\eps_k(x) - \frac{c_1}{2\pi} e_k = \frac{1}{|\log\eta|} \frac{c_1}{2\pi} (\log\left|\frac{x}{\eps}\right|) e_k + \frac{1}{|\log\eta|} \cD^\eta_T[g'](\frac{x}{\eps\eta}) + O\left(\frac{1}{|\log \eta|}\right).
\end{equation*}
To compute $\|v^\eps_k - \frac{c_1}{2\pi}e_k\|_{L^2(D)}^2$, we break the integrals into those on the cubes $Q_{\eps,i}$'s. Using the Poincar\'e inequality on $Q_{\eps,i}$, we get the following analog of \eqref{eq:vav-1}
\begin{equation*}
\int_{Q_{\eps,i}} |v^\eps_k - \frac{c_1}{2\pi}e_k|^2 \le C|\log\eta|^{-2} \eps^2\|\nabla \chi^\eta_k\|_{L^2(\frac1\eta\bT^d)}^2 + \eps^2\left|\langle v^\eps_k\rangle - \frac{c_1}{2\pi}e_k\right|^2.
\end{equation*}
The last term satisfies
\begin{equation*}
\left|\langle v^\eps_k\rangle - \frac{c_1}{2\pi}e_k\right| \le |\log\eta|^{-1}\left(\left|\langle \log|\frac{x}{\eps}|\rangle_{Q_{\eps,i}}\right| + |\langle \cD^\eta_T[g']\rangle_{\frac1\eta \bT^d\setminus \ol T} |\right) + C|\log\eta|^{-1}.
\end{equation*}
The term involving $\cD^\eta_T[g']$ is controlled exactly as before and its average is of order one. We compute
\begin{equation*}
\left|\langle \log|\frac{x}{\eps}|\rangle_{Q_{\eps,i}}\right| = \left|\langle \log|x|\rangle_{Q_1} \right| \le C.
\end{equation*}
We hence conclude that
\begin{equation*}
\left|\langle v^\eps_k\rangle - \frac{c_1}{2\pi}e_k\right| \le C|\log\eta|^{-1}.
\end{equation*}
Using those estimates together with \eqref{eq:chigradbdd} in \eqref{eq:vav-3}, we conclude that
\begin{equation*}
\|v^\eps_k - \frac{c_1}{2\pi}e_k\|^2_{L^2(Q_{\eps,i})} \le C\eps^2|\log \eta|^{-1}.
\end{equation*}
Again, this estimate is uniform for all cubes $Q_{\eps,i}$'s, and there are $O(\eps^{-2})$ many of them, and we hence conclude that
\begin{equation*}
\|v^\eps_k - \frac{c_1}{2\pi}e_k\|_{L^2(Q_{\eps,i})} \le C|\log \eta|^{-\frac12}.
\end{equation*}
This completes the proof.
\end{proof}

%%%%%%%%%%%
%%%%%%%%%%%
\section{A unified proof for qualitative homogenization}
\label{sec:qualh}

In this section, we prove Theorem \ref{thm:qualh} with a unified method. In view of the estimates \eqref{eq:uepsbdd} and \eqref{eq:uepsbddp}, the sequence $\{\tilde u^\eps/(1\wedge\sigma_\eps^2)\}$ and $\{\nabla \tilde u^\eps/(1\wedge \sigma_\eps)\}$ are uniformly bounded in $L^2$; here $a \wedge b$ means $\min\{a,b\}$.

Hence, in the super-critical setting, we can extract a subsequence that is still denoted by $\eps \to 0$, along which
\begin{equation*}
\frac{\tilde u^\eps}{\sigma_\eps^2} \to u \qquad \text{ weakly in } L^2(D).
\end{equation*}
In the critical and sub-critical settings, we can extract a subsequence along which
\begin{equation*}
\tilde u^\eps \to u \qquad \text{ weakly in } H^1_0(D).
\end{equation*}
The qualitative homogenization results amount to determining the limit $u$ and showing that the whole sequence converges.

In this section, we establish those results using the standard method of oscillating test functions. To start, let $\varphi \in C^\infty_c(D;\R)$ be a real valued test function with compact support in $D$. Along an aforementioned converging subsequence of $u^\eps$, test $\varphi v^\eps_k$, which belongs to $H^1_0(D^\eps)$, against the equation \eqref{eq:hetlame}, we get
\begin{equation*}
\begin{aligned}
\int_D \mu \varphi \nabla \tilde u^\eps : \nabla v^\eps_k + (\lambda+\mu) \varphi (\div \tilde u^\eps)(\div v^\eps_k) &+ \int_D \mu \nabla \tilde u^\eps : (\nabla \varphi \otimes v^\eps_k) + (\lambda + \mu) (\div \tilde u^\eps)(\nabla \varphi \cdot v^\eps_k)\\
&= \int_D \varphi (f\cdot v^\eps_k).
\end{aligned}
\end{equation*} 
On the other hand,  since $\varphi u^\eps$ belongs to $H^1(\eps\R^d_f)$, we can test it against equation \eqref{eq:vepseq}, and obtain
\begin{equation*}
\begin{aligned}
\int_D \mu \varphi \nabla v^\eps_k : \nabla \tilde u^\eps  + (\lambda+\mu) \varphi (\div v^\eps_k)(\div \tilde u^\eps) &+ \int_D \mu (\nabla\varphi \otimes \tilde u^\eps ): \nabla v^\eps_k+ (\lambda + \mu) (\nabla \varphi \cdot \tilde u^\eps)(\div v^\eps_k)\\
&= \int_D \varphi (e_k\cdot \frac{\tilde u^\eps}{\sigma^2_\eps}).
\end{aligned}
\end{equation*} 
Take the difference between those equations, we get the key identity
\begin{equation}
\label{eq:wibp}
\begin{aligned}
\int_D \mu \nabla \tilde u^\eps : (\nabla \varphi \otimes v^\eps_k) &+ \int_D (\lambda + \mu) (\div \tilde u^\eps)(\nabla \varphi \cdot v^\eps_k) - \int_D \mu (\nabla\varphi \otimes \tilde u^\eps ): \nabla v^\eps_k\\
&- \int_D (\lambda + \mu) (\nabla \varphi \cdot \tilde u^\eps)(\div  v^\eps_k) =\int_D \varphi (f\cdot v^\eps_k- e_k\cdot \frac{\tilde u^\eps}{\sigma^2_\eps})
\end{aligned}
\end{equation}
Let us name the five integrals in the identity above by $I_1, I_2, \dots, I_5$ in order of their appearance. We need to find their limits in each asymptotic regimes for $\sigma_\eps$. The trick of the procedure above is, the integral terms that involve products of a pair of weakly converging quantities, namely the integral of $\nabla \tilde u^\eps : \nabla v^\eps$, are all eliminated, and integrals that survived in \eqref{eq:wibp} only involve products of a weakly converging function with strongly converging ones. 

\subsection{The super-critical setting} We only address the dilute case. In this setting, $\sigma_\eps$ converges to zero, and along the converging subsequence, $\tilde u^\eps/\sigma^2_\eps \to u$ weakly in $L^2$, and $\nabla \tilde u^\eps$ is of order $O(\sigma_\eps)$. Inspecting the integrals in \eqref{eq:wibp}, we find, using \eqref{eq:uepsbddp}, \eqref{eq:vgradbdd} and \eqref{eq:vstrong}, as $\eps \to 0$,
\begin{equation*}
\begin{aligned}
|I_1| &\le C\|\nabla \varphi\|_{L^\infty}\|\nabla \tilde u^\eps\|_{L^2}\|v^\eps_k\|_{L^2} \le C\sigma_\eps \to 0,\\
|I_2| &\le C\|\nabla\tilde u^\eps\|_{L^2} \|\nabla \varphi\|_{L^\infty}\|v^\eps_k\|_{L^2} \le C\sigma_\eps \to 0,\\
|I_3| + |I_4| &\le C\|\tilde u^\eps\|_{L^2} \|\varphi\|_{L^\infty}\|\nabla v^\eps_k\|_{L^2} \le C\sigma_\eps,\\
I_5 &\to \int_D \varphi(M^{-1}f - u)\cdot e_k.
\end{aligned}
\end{equation*}
In the limit of $I_5$, we also used the fact that $M^{-1}$ is symmetric. As a result, passing $\eps \to 0$ in \eqref{eq:wibp}, we get
\begin{equation*}
\int_D \varphi(M^{-1}f - u)\cdot e_k = 0,
\end{equation*}
which holds for all test function $\varphi$ and for all $k = 1,\dots,d$. It follows that
\begin{equation*}
u = M^{-1}f.
\end{equation*}
The above formula dictates the possible limit of $\tilde u^\eps/\sigma^2_\eps$. Hence, the whole sequence converges to this $u$. This completes the proof in the super-critical setting.

\subsection{The critical setting}   

In this setting, $\sigma_\eps \to \sigma_0$ for some $\sigma_0 \in (0,\infty)$, and along a converging subsequence, $\tilde u^\eps \to u$ weakly in $H^1_0(D)$. By the Rellich's lemma, we also have $\tilde u^\eps \to u$ strongly in $L^2$, and $\nabla \tilde u^\eps \to \nabla u$ weakly in $L^2$.

We examine the integrals in \eqref{eq:wibp}, and by using the weak convergence of $\nabla \tilde u^\eps$ and $\nabla v^\eps_k$, together with the strong convergence of $\tilde u^\eps$ and $v^\eps_k$, we deduce that, by sending $\eps \to 0$,
\begin{equation*}
\begin{aligned}
\int_D \mu \nabla u : (\nabla \varphi\otimes M^{-1}e_k)  + \int_D (\lambda+\mu)(\div u)(M^{-1}\nabla \varphi \cdot e_k) = \int_D \varphi(M^{-1}f - \frac{u}{\sigma_0^2})\cdot e_k.
\end{aligned}
\end{equation*}
We emphasize that the limit of $I_3$ and $I_4$ vanishes because $\nabla v^\eps_k$ weakly converges to zero. Using integration by parts, we can recast the above as
\begin{equation*}
-\int_D \varphi \left(\mu \Delta u + (\lambda+\mu)\nabla \div u\right) \cdot M^{-1}e_k =  \int_D \varphi(M^{-1}f - \frac{u}{\sigma_0^2})\cdot e_k.
\end{equation*}
Since $M^{-1}$ is symmetric, we can move $M^{-1}$ on the left hand side to the front of $\cL^{\lambda,\mu}[u]$. Then we multiply $M$ on both sides to get
\begin{equation*}
-\int_D \varphi \left(\mu \Delta u + (\lambda+\mu)\nabla \div u\right) \cdot e_k =  \int_D \varphi(f - \frac{Mu}{\sigma_0^2})\cdot e_k.
\end{equation*}
This holds for all test functions $\varphi$ and for all $k \in \{1,\dots,d\}$. We conclude that
\begin{equation*}
-\cL^{\lambda,\mu} u + \frac{M}{\sigma_0^2}u = f \qquad \text{in distribution in $D$.}
\end{equation*}
Since we already have $u \in H^1_0(D)$, $u$ is the unique weak solution to \eqref{eq:homcr}. This determines the possible limit of $\tilde u^\eps$ uniquely and, hence, the whole sequence converges.

\subsection{The sub-critical setting} In this setting, $\sigma_\eps \to \infty$, and along a converging subsequence, $\tilde u^\eps \to u$ weakly in $H^1_0(D)$. We can argue almost exactly as in the previous setting. We point out two differences. Firstly, the term in $I_5$ involving $\sigma_\eps$ vanishes in the limit. Secondly, $I_3$ and $I_4$ vanish in the limit for a reason different from the previous settings, namely due to \eqref{eq:vgradbdd}. It follows that the only limit $u$ for $\tilde u^\eps$ is given by the solution to
\begin{equation*}
-\cL^{\lambda,\mu} [u] = f, \qquad \text{in } D
\end{equation*}
with $u \in H^1_0$. As a result, the whole sequence converges to this limit.

We also emphasize that our approach is uniform with respect to all the asymptotic regimes of $\sigma_\eps$ and for all $d\ge 2$. The necessary modifications for $d = 2$ is encoded in  the asymptotic analysis of $v^\eps_k$'s, and the matrix $M$ is defined accordingly.

%%%%%%%%%%%
%%%%%%%%%%%

\section{Correctors and error estimates}
\label{sec:quanth}

Another feature of our approach is that the method yields natural correctors and error estimates, with inspirations from the informal two-scale expansion method. We prove Theorem \ref{thm:quanh} in this section.

%%%%%%
\subsection{Super-critical setting} We only consider the dilute case. For the super-critical setting, $\sigma_\eps$ is a small number. By rescaling the corrector suggested by the formal two-scale expansion, we should consider the discrepancy function
\begin{equation*}
\zeta^\eps = \frac{u^\eps}{\sigma^2_\eps} - f^k(x)v^\eps_k(x).
\end{equation*}
Note that $\xi^\eps \in H^1_0(D^\eps)$ and we set its value as zero inside the holes. Direct computation shows that
\begin{equation*}
 \begin{aligned}
 -\cL^{\lambda,\mu} [\zeta^\eps] =  \mu\left[v^\eps_k \Delta f^k + 2 \partial_\ell f^k \partial_\ell v^\eps_k\right] &+ (\lambda+\mu)\left[\partial_i(v^\eps_k)^\ell \partial_\ell f^k +  (\partial^2 f^k)v^\eps_k \right]\\
 &+ (\lambda+\mu)(\div v^\eps_k) \nabla f^k, \qquad \text{in } D^\eps.
\end{aligned}
\end{equation*} 
Here $\partial^2 f^k$ denotes the second order derivative matrix of $f^k$. We assume that $f \in W^{2,d}(D)$ so that the right hand side is an $L^2$ function and the equation is satisfied in the weak sense. Test $\zeta^\eps$ against this equation, we obtain
\begin{equation}
\label{eq:quant-1}
\begin{aligned}
&\mu\|\nabla\zeta^\eps\|^2_{L^2} + (\lambda+\mu)\|\div \zeta^\eps\|^2_{L^2} = \int_D \mu \zeta^\eps \cdot v^\eps_k \Delta f^k + (\lambda+\mu)\zeta^\eps \cdot [ (\partial^2  f^k)v^\eps_k]\\
&\qquad + (\lambda+\mu)\left[\int_D \div( v^\eps_k - M^{-1}e_k)\zeta^\eps \cdot \nabla f^k  + \int_D (\zeta^\eps)^i \partial_i (v^\eps_k - M^{-1}e_k)^\ell \partial_\ell f^k \right]\\
&\qquad  +2\mu\int_D \zeta^\eps \cdot[\partial_\ell f^k (\partial_\ell (v^\eps_k - M^{-1}e_k))].
\end{aligned}
\end{equation}
Let us label the four integrals on the right hand side as $I_1,\dots,I_4$. Note that in $I_2, I_3$ and $I_4$ we inserted the constant $M^{-1}e_k$ inside some derivatives without violating the equation. Assume $d\ge 3$ for the moment and set $p = 2d/(d-2)$. The first integral is then controlled by
\begin{equation}
\label{eq:quant-2}
|I_1| \le C\|\partial^2 f\|_{L^d}\|v^\eps\|_{L^p}\|\zeta^\eps\|_{L^2} \le C\sigma_\eps \|\partial^2 f\|_{L^d}\|v^\eps\|_{L^p}\|\nabla\zeta^\eps\|_{L^2}.
\end{equation}
For the rest of the integrals, we need to perform an integration by parts (in $D^\eps$, and, note that $\zeta^\eps \in H^1_0(D^\eps)$) first to shift the derivatives off $v^\eps$ terms. For $I_2$, the following holds.
\begin{equation*}
\begin{aligned}
I_2 &= -\int_D (v^\eps_k - M^{-1}e_k)^\ell \left( \partial_\ell (\zeta^\eps)^i \partial_i f^k + (\zeta^\eps)^i \partial_i \partial_\ell f^k\right)\\
&= -\int_D (v^\eps_k - M^{-1}e_k) \cdot (\nabla \zeta^\eps)^T \nabla f^k + (v^\eps_k - M^{-1}e_k)\cdot(\partial^2 f^k)\zeta^\eps.
\end{aligned}
\end{equation*}
We deduce that
\begin{equation}
\label{eq:quant-3}
\begin{aligned}
|I_2| \le& \sum_k \|v^\eps_k - M^{-1}e_k\|_{L^p}\left(\|\nabla f\|_{L^d}\|\nabla \zeta^\eps\|_{L^2} + \|\partial^2 f\|_{L^d} \|\zeta^\eps\|_{L^2}\right)\\
&\le C\eta^{\frac{d-2}{2}}(1+\sigma_\eps)\|f\|_{W^{2,d}}\|\nabla \zeta^\eps\|_{L^2}.
\end{aligned}
\end{equation}
The integrals $I_3$ and $I_4$ can be treated in the same manner and they satisfy the same bound above. Using \eqref{eq:quant-2} and \eqref{eq:quant-3} in \eqref{eq:quant-1}, we finally get
\begin{equation*}
\|\nabla \zeta^\eps\|_{L^2} \le C\left(\sigma_\eps + \eta^{\frac{d-2}{2}}\right)\|f\|_{W^{2,d}}.
\end{equation*}
By the Poincar\'e inequality, we also have
\begin{equation*}
\|\zeta^\eps\|_{L^2} \le C\left(\sigma^2_\eps + \eps \right)\|f\|_{W^{2,d}}.
\end{equation*}
This is the desired estimate for $d \ge 3$.

In the case of $d=2$, we only need to replace $p$ by $2$ and use $W^{2,\infty}$ control on $f$. The arguments above then follow and we get
\begin{equation*}
\|\nabla \zeta^\eps\|_{L^2} \le C\left(\sigma_\eps + |\log \eta|^{-\frac12}\right)\|f\|_{W^{2,d}},
\end{equation*} 
and
\begin{equation*}
\|\zeta^\eps\|_{L^2} \le C\left(\sigma^2_\eps + \eps \right)\|f\|_{W^{2,d}}.
\end{equation*}

%%%%%%
\subsection{The critical setting} In this setting, $\sigma_\eps$ is of order one, and $\sigma_\eps \to \sigma_0$ as $\eps \to 0$. We consider the discrepancy function
\begin{equation*}
\zeta^\eps = u^\eps - \sigma^2_\eps (\frac{M}{\sigma_0^2}u)^k v^\eps_k.
\end{equation*}
We emphasize that $\zeta^\eps \in H^1_0(D^\eps)$. This can be seen as an analog of the discrepency used in the previous setting, except that  we replace $f$ by $\frac{M}{\sigma_0}u$. Direct computation then shows
\begin{equation*}
\begin{aligned}
-\cL^{\lambda,\mu} &[\zeta^\eps] = f - \frac{M}{\sigma_0^2}u + \mu \frac{\sigma_\eps^2}{\sigma^2_0}\left[v^\eps_k \Delta (Mu)^k + 2 \partial_\ell (Mu)^k \partial_\ell v^\eps_k\right] \\
&+ (\lambda+\mu)\frac{\sigma_\eps^2}{\sigma^2_0}\left[\partial_i(v^\eps_k)^\ell \partial_\ell (Mu)^k +  (\partial^2 (Mu)^k)v^\eps_k \right]
 + (\lambda+\mu)\frac{\sigma_\eps^2}{\sigma^2_0}(\div v^\eps_k) \nabla (Mu)^k
 \end{aligned}
  \qquad \text{in } D^\eps.
\end{equation*}
Using \eqref{eq:homcr} and by some algebraic manipulations, we can rewrite the above as
\begin{equation*}
\begin{aligned}
-\frac{\sigma_0^2}{\sigma^2_\eps} \cL^{\lambda,\mu} [\zeta^\eps] =& \left(\frac{\sigma^2_0}{\sigma^2_\eps}-1\right)\cL^{\lambda,\mu}u +  \mu (\Delta (Mu)^k)(v^\eps_k-M^{-1}e_k)  +  (\lambda+\mu) (\partial^2 (Mu)^k)(v^\eps_k - M^{-1}\cdot e_k) \\
&+ 2\mu \partial_\ell (Mu)^k \partial_\ell v^\eps_k + (\lambda+\mu)\left[\partial_i(v^\eps_k)^\ell \partial_\ell (Mu)^k 
 + (\div v^\eps_k) \nabla (Mu)^k\right] \qquad \text{in } D^\eps.
 \end{aligned}
\end{equation*}
After replacing $\nabla v^\eps_k$ by $\nabla(v^\eps_k - M^{-1}e_k)$, we test $\zeta^\eps$ against the equation and obtain
\begin{equation}
\label{eq:quant-4}
\begin{aligned}
&\frac1C \|\nabla\zeta^\eps\|^2_{L^2} \le \int_D \mu \zeta^\eps \cdot (v^\eps_k-M^{-1}e_k) \Delta (Mu)^k + (\lambda+\mu)\zeta^\eps \cdot [ (\partial^2  (Mu)^k)(v^\eps_k-M^{-1}e_k)]\\
&\qquad + (\lambda+\mu)\left[\int_D \div( v^\eps_k - M^{-1}e_k)\zeta^\eps \cdot \nabla (Mu)^k  + \int_D (\zeta^\eps)^i \partial_i (v^\eps_k - M^{-1}e_k)^\ell \partial_\ell (Mu)^k \right]\\
&\qquad  +2\mu\int_D \zeta^\eps \cdot[\partial_\ell (Mu)^k (\partial_\ell (v^\eps_k - M^{-1}e_k))] + |\sigma^2_\eps - \sigma^2_0|\sigma^{-2}_\eps \left|\int_{D^\eps} \zeta^\eps \cdot \cL^{\lambda,\mu}u\right|.
\end{aligned}
\end{equation}
The first four integrals on the right hand side of the inequality above can be controlled as before, and, for $d\ge 3$, they are bounded by
\begin{equation*}
C\sum_k \|v^\eps_k - M^{-1}e_k\|_{L^p} \|Mu\|_{W^{2,d}}\|\nabla \zeta^\eps\|_{L^2} \le C\eta^{\frac{d-2}{2}}\|\nabla \zeta^\eps\|_{L^2}.
\end{equation*}
The last integral can be recognized as the bilinear form associated to the Lam\'e system evaluated at the pair $(\zeta^\eps,u)$, and hence the last term in \eqref{eq:quant-4} is bounded by
\begin{equation*}
C|\sigma^2_\eps-\sigma^2_0|\|\nabla u\|_{L^2}\|\nabla \zeta\|_{L^2}.
\end{equation*}
Combine the above estimates, we obtain
\begin{equation*}
\|\nabla \zeta^\eta\|_{L^2} + \|\zeta\|_{L^2} \le C\left(\eps + |\sigma^2_\eps-\sigma^2_0|\right)\|u\|_{W^{2,d}}.
\end{equation*}
For $d=2$, the above estimate still holds if we use $W^{2,\infty}$ estimate for $u$ instead.

%%%%%%
\subsection{The sub-critical setting} In this setting, $\sigma_\eps \to \infty$ and hence $\sigma^{-1}_\eps$ is a small number. We consider the discrepancy function
\begin{equation*}
\zeta^\eps = u^\eps - (Mu)^k v^\eps_k,
\end{equation*}
which belongs to $H^1_0(D^\eps)$, and we set its value as zero in the holes. Computation shows
\begin{equation*}
\begin{aligned}
-\cL^{\lambda,\mu} &[\zeta^\eps] = f - \frac{Mu}{\sigma^2_\eps} + \mu \left[v^\eps_k \Delta (Mu)^k + 2 \partial_\ell (Mu)^k \partial_\ell v^\eps_k\right] \\
&+ (\lambda+\mu)\left[\partial_i(v^\eps_k)^\ell \partial_\ell (Mu)^k +  (\partial^2 (Mu)^k)v^\eps_k \right]
 + (\lambda+\mu)(\div v^\eps_k) \nabla (Mu)^k
 \end{aligned}
  \qquad \text{in } D^\eps.
\end{equation*}
Using the equation satisfied by $u$, we rewrite the above as
\begin{equation*}
\begin{aligned}
-\cL^{\lambda,\mu} [\zeta^\eps] =& -\frac{Mu}{\sigma^2_\eps} +  \mu (\Delta (Mu)^k)(v^\eps_k-M^{-1}e_k)  +  (\lambda+\mu) (\partial^2 (Mu)^k)(v^\eps_k - M^{-1}\cdot e_k) \\
&+ 2\mu \partial_\ell (Mu)^k \partial_\ell (v^\eps_k-M^{-1}e_k) + (\lambda+\mu)\partial_i(v^\eps_k - M^{-1}e_k)^\ell \partial_\ell (Mu)^k\\
&+ (\lambda+\mu)\div (v^\eps_k-M^{-1}e_k) \nabla (Mu)^k
 \end{aligned}
  \qquad \text{in } D^\eps.
\end{equation*}
Test $\zeta^\eps$ against this equation, we obtain
\begin{equation}
\label{eq:quant-5}
\begin{aligned}
&\mu \|\nabla\zeta^\eps\|^2_{L^2} \le \int_D \mu \zeta^\eps \cdot (v^\eps_k-M^{-1}e_k) \Delta (Mu)^k + (\lambda+\mu)\zeta^\eps \cdot [ (\partial^2  (Mu)^k)(v^\eps_k-M^{-1}e_k)]\\
&\qquad +\int_D  (\lambda+\mu)\left[\div (v^\eps_k-M^{-1}e_k)\zeta^\eps \cdot \nabla (Mu)^k  +(\zeta^\eps)^i (\partial_i (v^\eps_k-M^{-1}e_k)^\ell) \partial_\ell (Mu)^k \right]\\ &\qquad + 2\mu\int_D \zeta^\eps \cdot[\partial_\ell (Mu)^k (\partial_\ell (v^\eps_k-M^{-1}e_k))]- \frac1{\sigma^2_\eps} \int_D \zeta^\eps\cdot Mu.
\end{aligned}
\end{equation}
The first three integrals on the right hand side can be analyzed as before and, for $d\ge 3$, they are bounded by
\begin{equation*}
C\sum_k \|v^\eps_k - M^{-1}e_k\|_{L^p}\|Mu\|_{W^{2,d}}\|\nabla \zeta^\eps\|_{L^2} \le C\eta^{\frac{d-2}{2}}\|Mu\|_{W^{2,d}}\|\nabla \zeta^\eps\|_{L^2}.
\end{equation*}
Note that we also use the usual Poincar\'e inequality on $D$ as $\zeta^\eps \in H^1_0(D)$. Using H\"older inequality and the usual Poinca\'e inequality, we can bound the last integral from above by
\begin{equation*}
C\sigma^{-2}_\eps \|Mu\|_{L^2}\|\nabla \zeta^\eps\|.
\end{equation*}
Combine those results, we deduce that, for $d\ge 3$,
\begin{equation*}
\|\nabla \zeta^\eps\|_{L^2} + \|\zeta^\eps\|_{L^2} \le \left(\eta^{\frac{d-2}{2}} + \sigma_\eps^{-2}\right)\|u\|_{W^{2,d}}.
\end{equation*}
For $d=2$, this estimate holds with $\eta^{\frac{d-2}{2}}$ replaced by $|\log \eta|^{-\frac12}$ and with $W^{2,d}$ replaced by $W^{2,\infty}$.

%%%%%%%%
\section*{acknowledgements}
The author would like to thank Xin Fu for helpful discussions on layer potentials for Lam\'e systems.

%%%%%
\appendix

\section{Some useful lemmas}

The following results are very helpful and have been used in the main parts of the paper. 

\begin{theorem}[A Poincar\'e inequality]\label{thm:poincare} Let $d\ge 2$. Let $r, R$ be two positive real numbers and $r < R$. Then there exists a constant $C > 0$ that depends only on the dimension $d$, such that for any $u \in H^1(B_R(0))$ satisfying $u = 0$ in $B_r(0)$, we have
\begin{equation}
\label{eq:poincare}
\|u\|_{L^2(B_R)} \le \begin{cases}
CR(\frac{r}{R})^{-\frac{d-2}{2}} \|\nabla u\|_{L^2(B_R)}, &\qquad d\ge 3,\\
CR|\log(\frac{r}{R})|^\frac12 \|\nabla u\|_{L^2(B_R)}, &\qquad d=2.
\end{cases}
\end{equation}
\end{theorem}
We refer to \cite[Lemma 3.4.1]{Allaire91-2} or \cite[Theorem A.1]{Jing20} for the proof. This inequality accounts for the various asymptotic regimes for \eqref{eq:hetlame} depending on the relative smallness of $\eta$ with respect to $\eps$. Clearly, if we change one or both of the balls to cubes, the above inequality still holds. In particular, it can be applied on the $\eps$-cubes, $\eps(z+\ol Y_f)$, $z\in \Z^d$, which form $\eps \R^d_f$ and $D^\eps$. 

\begin{lemma}\label{lem:semifred}
Suppose $H$ is a Hilbert space and $\mathcal{T} : H\to H$ is a bounded linear operator on $H$ and $\mathcal{T}^*$ is the adjoint operator. Suppose $\mathcal{T}$ has closed range, $\ker(\mathcal{T})$ has finite dimension $k$, and, moreover, $\mathcal{T} - \mathcal{T}^*$ is compact. Then $\dim \ker(\mathcal{T}^*) = k$ as well.
\end{lemma}

This is rephrased from Lemma 2.3 of \cite{FabKenVer}. It can be proved directly, or, by using the fact that $\mathcal{T}$ is semi-Fredholm and that semi-Fredholmness and the index of such an operator are preserved by compact perturbations.



\end{document}